\documentclass[a4paper,12pt]{amsart}

\usepackage[utf8]{inputenc}
\usepackage[T1]{fontenc}
\usepackage{lmodern}
\usepackage[english]{babel}
\usepackage{microtype}

\usepackage{amsmath,amssymb,amsfonts,amsthm}
\usepackage{mathtools,accents}
\usepackage{mathrsfs}
\usepackage{aliascnt}
\usepackage{braket}
\usepackage{bm}
\usepackage{esint}
\usepackage{comment}

\usepackage[a4paper,margin=3cm]{geometry}
\usepackage[citecolor=blue,colorlinks]{hyperref}

\usepackage{enumerate}
\usepackage{xcolor}

\makeatletter
\g@addto@macro\@floatboxreset\centering
\makeatother

\makeatletter
\def\newaliasedtheorem#1[#2]#3{
  \newaliascnt{#1@alt}{#2}
  \newtheorem{#1}[#1@alt]{#3}
  \expandafter\newcommand\csname #1@altname\endcsname{#3}
}
\makeatother

\numberwithin{equation}{section}

\newtheoremstyle{slanted}{\topsep}{\topsep}{\slshape}{}{\bfseries}{.}{.5em}{}

\theoremstyle{plain}
\newtheorem{theorem}{Theorem}[section]
\newaliasedtheorem{proposition}[theorem]{Proposition}
\newaliasedtheorem{claim}[theorem]{Claim}
\newaliasedtheorem{lemma}[theorem]{Lemma}
\newaliasedtheorem{corollary}[theorem]{Corollary}
\newaliasedtheorem{counterexample}[theorem]{Counterexample}

\theoremstyle{definition}
\newaliasedtheorem{definition}[theorem]{Definition}
\newaliasedtheorem{question}[theorem]{Question}
\newaliasedtheorem{example}[theorem]{Example}
\newaliasedtheorem{openproblem}[theorem]{Open problem}

\theoremstyle{remark}
\newaliasedtheorem{remark}[theorem]{Remark}

\newcommand{\setC}{\mathbb{C}}
\newcommand{\setN}{\mathbb{N}}

\newcommand{\setR}{\mathbb{R}}

\newcommand{\cE}{\mathcal{E}}
\newcommand{\cV}{\mathcal{V}}
\newcommand{\cB}{\mathcal{B}}

\newcommand{\eps}{\varepsilon}

\let\phi\varphi

\newcommand{\Id}{\mathrm{Id}}

\newcommand{\di}{\mathop{}\!\mathrm{d}}

\newcommand{\Ch}{{\sf Ch}}
\newcommand{\scal}[2]{\ensuremath{\langle #1 , #2 \rangle}} 

\DeclareMathOperator{\Lip}{Lip}

\newcommand{\haus}{\mathscr{H}}

\newcommand{\dist}{\mathsf{d}}

\newcommand{\meas}{\mathfrak{m}}
\newcommand{\dimnew}{{\rm dim}_{\dist,\meas}}

\newcommand{\diam}{\mathrm{diam}}
\DeclareMathOperator{\CD}{CD}
\DeclareMathOperator{\RCD}{RCD}

\DeclareMathOperator{\vol}{\mathrm{vol}}

\DeclareMathOperator{\Scal}{Scal}
\DeclareMathOperator{\Ric}{Ric}

\newfont{\tmpf}{cmsy10 scaled 2500}

\DeclareMathOperator{\ext}{xt}

\title[Spectral embeddings: a survey]{A survey on spectral embeddings and their application in data analysis}
\author{David Tewodrose}
\address{CY Cergy Paris University, Laboratoire de mathématiques AGM, UMR CNRS 8088, 2 av. Adolphe Chauvin, 95302 Cergy-Pontoise cedex}
\email{david.tewodrose@cyu.fr}

\begin{document}

\maketitle


\begin{abstract}
The aim of this survey is to present some aspects of the Bérard-Besson-Gallot spectral embeddings of a closed Riemannian manifold from their origins in Riemannian geometry to more recent applications in data analysis.
\end{abstract}

\tableofcontents

\section{Introduction}

The spectral embeddings we deal with in this survey were introduced by P.~Bérard, G.~Besson and S.~Gallot in \cite{BerardBessonGallot} to provide new distances on the set of isometry classes of closed Riemannian manifolds -- here and in the sequel, by closed we mean compact without boundary, and we tacitly assume the manifolds to be smooth and connected with smooth Riemannian metric. Three families of embeddings of a given closed Riemannian manifold $(M,g)$ were proposed in this article:
\begin{itemize}
\item the unrescaled spectral embeddings $I_t^a : M \to l^2$,

\item the rescaled spectral embeddings $\Psi_t^a : M \to l^2$,

\item the spherical spectral embeddings $K_t^a : M \to S^{\infty}$,
\end{itemize}
where $l^2$ is the Hilbert space of square summable real-valued sequences, $S^{\infty}$ is the unit sphere in $l^2$, and $t>0$ is a parameter. These embeddings all depend on the choice of an orthonormal basis $a$ of $L^2(M)$ made of eigenfunctions of (minus) the Laplace-Beltrami operator $-\Delta_{(M,g)}$ of $(M,g)$. They are called \textit{spectral} because they are defined via the spectrum of $-\Delta_{(M,g)}$. For instance, the rescaled embeddings are defined by
$$
\Psi_t^{a}\,  : \, \, \,  
\begin{cases}
\,\, M & \to \quad l^2\\
\,\, x & \mapsto \quad s_{n,t} (e^{-\lambda_i t /2} \phi_i(x))_{i \ge 1}
\end{cases}
$$
where $a=\{\phi_i\}_{i \ge 0}$ is an orthonormal basis of $L^2(M)$ made of eigenfunctions of $-\Delta_{(M,g)}$, the numbers $0=\lambda_0<\lambda_1 \le \lambda_2 \le \ldots \to + \infty$ are the corresponding eigenvalues, and $s_{n,t}$ is a rescaling factor.

The unrescaled spectral embeddings revealed more interesting in regard to the authors' original motivations because they led to the so-called spectral distances, see Section 2.1 and 2.2. However, the rescaled ones have inspired remarkable developments in data analysis \cite{BelkinNiyogi, CoifmanLafon, BelkinNiyogi2, SingerWu} which, in turn, have raised deep theoretical questions \cite{Bates, Portegies, Wu, AmbrosioHondaPortegiesTewodrose}.\\

The main interest in the maps $\Psi_t^a$ lies in the fact that they are almost isometric, in the sense that they imply pull-back Riemannian metrics $g_t$ on $M$ such that
$$
g_t = g + t A + o(t), \qquad t \downarrow 0,
$$
in the sense of pointwise convergence, where $A$ is a $(0,2)$-tensor depending only on the Ricci and scalar curvature of $(M,g)$. This is Theorem 5 in \cite{BerardBessonGallot}. Of course, a celebrated result of J.~Nash \cite{Nash, DeLellis} implies that $(M,g)$ can be smoothly and isometrically embedded into $\setR^D$ for some $D$ depending only on the dimension of $M$. But Nash's embedding is less convenient than the Bérard-Besson-Gallot ones when analysis on the manifold is the matter: indeed, any reasonable piece of information on the eigenvalues or eigenfunctions of $\Delta_{(M,g)}$ directly affects the spectral embeddings, but not Nash's embedding. Nevertheless, the target space of the Bérard-Besson-Gallot embeddings is a Hilbert space, while the target space of Nash's one is Euclidean. This is one of the main reasons -- numerical applications being the other ones -- why truncated versions of the maps $\Psi_t^a$, that is maps of the form
$$
[\Psi_t^{a}]^m\,  : \, \, \,  
\begin{cases}
\,\, M & \to \quad \setR^m\\
\,\, x & \mapsto \quad s_{n,t} (e^{-\lambda_i t /2} \phi_i(x))_{1 \le i \le m}
\end{cases}
$$
where $m \in \setN \backslash \{0\}$, came under the spotlight. The works of J. Bates \cite{Bates} and J. Portegies \cite{Portegies} showed that under certain geometric constraints, namely Ricci curvature and injectivity radius bounded from below together with prescribed total volume, there exists a dimension $\overline{m}$ depending on the constraints such that $[\Psi_t^a]^{\overline{m}}$ is a smooth embedding for any $t>0$. We explain this result in Section 2.3. J.~Portegies went further by showing that under the same geometric constraints, the maps $[\Psi_t^a]^{m}$ can be made arbitrary close to being an isometry in the following sense: for any $\eps>0$, there exists $t_o>0$ such that for any $t \in (0,t_o)$, there exists $\overline{m}$ depending on $t$, $\eps$, and the geometric constraints, such that
$$
1-\eps < |D_x [\Psi_t^a]^{m}| < 1+\eps
$$ 
for any $x \in M$ and $m \ge \overline{m}$. We return to this quantitative statement in Section 2.4.

In \cite{SingerWu}, in the search for substantial improvements of the spectral embedding based techniques in data visualization, A.~Singer and H.-T.~Wu introduced another family of unrescaled spectral embeddings
$$
V_t^a : M \to l^2, \qquad t>0,
$$
by considering the connection Laplacian for vector fields $\Delta_{C, (M,g)}$ and its spectrum in place of the Laplace-Beltrami operator $\Delta_{(M,g)}$, and they showed that the natural associated distances
$$
\dist_{VDM, t}(x,y)=\|V_t(x)-V_t(y)\|_{l^2}
$$
on $M$ relates in a transparent way with the original Riemannian distance $\dist_g$ when $t \downarrow0$, see 
Theorem \ref{th:short-time}. Later on, H.-T.~Wu studied in \cite{Wu} the spectral distances naturally induced by the embeddings $V_t^a$ on the set of isometry classes of closed Riemannian manifolds and got precompactness results for these distances. We present this interesting alternative in Section 2.5.

Finally, in Section 2.6 we deal with recent results from  \cite{AmbrosioHondaPortegiesTewodrose} where the Bérard-Besson-Gallot spectral embeddings are studied in the context of compact metric measure spaces satisfying the synthetic Riemannian Curvature-Dimension condition $\RCD(K,N)$. Hopefully this may lead to new types of machine learning algorithms using approximations of a data set by means of a singular metric measure space instead of a smooth Riemannian manifold.\\

Let us spend more words on the link between these embeddings and data analysis. A classical problem in data science is the following. Say that we are given a phenomenon that we want to observe through several features. We repeat this phenomenon a high number of times and collect at each occurence some data related to the features. The data thus obtained can be stored in a matrix where each column corresponds to one occurence of the phenomenon and each row accounts for one feature. If $D$ is the number of features and $N$ the number of occurences, we get this way a $D\times N$ matrix. Each column of the matrix can be understood as a vector : this yields to a point cloud in $\setR^D$ made of $N$ elements. The structure of this point cloud reveals some information. A notable piece of information is the number $d$ of constraints the phenomenon may truly be subject to. Finding $d$ is of crucial importance, especially when it is small compared to $D$. Indeed, in most cases, $D$ is too high to perform direct classifying tasks in a reasonable amount of time on the original $D$-dimensional representation of the data set. This is why we seek for a faster to handle $d$-dimensional representation. This problem is usually called \textit{dimensionality reduction}. For more about this topic, we refer the non-expert reader to the introduction of
 \cite{Young} where three concrete applications of dimensionality reduction are presented: one from political science, one from psychology, and the last one from sociology.
 
Let us provide another popular example. Say that we want to classify pictures of a same object taken under various angle and light conditions. If the number of pixels used for the pictures is very high, say $64\times 64 = 4096$, then we are left with classifying points in $\setR^{4096}$, each coordinate corresponding to the brightness on one pixel. However, the observation of the object depends solely on three constraints: two for the angle conditions, one for the light conditions. In this case, a dimensionality reduction looks particularly relevant.
 
Classical techniques like Principal Component Analysis or Multidimensional Scaling are very effective regarding this issue, but only when the data lies on or near a low-dimensional \textit{linear} subspace $V$ of the feature space $\setR^D$, in which case finding the dimension $d$ of $V$ is part of the algorithm. Problem is that they fail to detect the geometry or the true dimensionality of a data set subject to \textit{non-linear} constraints, like in the previous example where the angle under which the pictures are taken obeys a spherical constraint. For this reason, the past twenty years have registered the development of new algorithms taking into account the possibly curved geometry of the data point cloud: without claiming to be exhaustive, let us cite Locally Linear Embedding \cite{RoweisSaul}, ISOMAP \cite{TenenbaumSilvaLangford}, Laplacian Eigenmaps \cite{BelkinNiyogi}, Diffusion Maps \cite{CoifmanLafon} and Vector Diffusion Maps \cite{SingerWu}. All these algorithms fall within the framework of \textit{manifold learning}, that is the study of data whose point cloud representation in $\setR^D$ lies on (or near) a smooth submanifold $M^d$ of $\setR^D$.

In Section 3. we describe the Laplacian Eigenmaps and the Vector Diffusion Maps algorithms for which the parallel with the Bérard-Besson-Gallot spectral embeddings is blatant. Indeed, the rough idea behind Laplacian Eigenmaps is to produce from the data point cloud a family of weighted graphs $(\cV,\cE,w_t)$ serving as discrete approximations of $M$ and to study suitable operators $-L_t$ on $(\cV,\cE,w_t)$. These operators can be viewed as discrete approximations of the Laplace-Beltrami operator $\Delta_M$ of $M$. Thus the eigenvalues $\lambda_i^{t}$ and eigenvectors $v_i^{t}$ of $L_t$ are reasonable approximations of the eigenvalues $\lambda_i$ and eigenfunctions $\phi_i$ of $-\Delta_M$. From this observation, one can define a discrete and finite counterpart ot the Bérard-Besson-Gallot embeddings for the data point cloud in order to embed it almost isometrically into some Euclidean space with low dimension.

\smallskip\noindent
\textbf{Acknowledgement.}
The author would like to thank the anonymous referee for precious remarks and especially for having raised Open Problem 2.

\section{Theoretical aspects}

In this section, we present the theoretical motivations behind the Bérard-Besson-Gallot embeddings, and list several important refinments.

\subsection{Motivations from Riemannian geometry}

Spectral embeddings were introduced by P.~Bérard, G.~Besson and S.~Gallot in \cite{BerardBessonGallot} to tackle specific problems from Riemannian geometry. These problems were described in \cite[VI.~E.]{Berard}. We provide here a brief summary.

In his celebrated book \cite{Gromov}, M.~Gromov introduced a distance on the set of compact metric spaces, nowadays called Gromov-Hausdorff distance, defined as follows: for two
compact metric spaces $(X,\dist_X)$ and $(Y,\dist_Y)$, set
$$
\dist_{GH}((X,\dist_X),(Y,\dist_Y)):=\inf_{(Z,i,j)}\left\{ \dist_{H,Z} (i(X),j(Y))\right\},
$$
where the infimum is taken over the set of triples $(Z,i,j)$ such that $Z$ is a metric space and $i : X \hookrightarrow Z$, $j : Y \hookrightarrow Z$ are isometric embeddings, and $\dist_{H,Z}$ stands for the Hausdorff distance in $Z$. This distance is particularly relevant when restricted to special classes of Riemannian manifolds satisfying geometric constraints. For instance, the set
$$
\mathcal{N}(n,S,D,V) : = \{ (M^n,g) \, \,\text{closed  $\, : \,\, |\mathrm{Sect}| \le S, \,\,\mathrm{diam}(M) \le D, \,\,\vol(M) \ge V$}\},
$$
where $\mathrm{Sect}$ and $\mathrm{diam}$ denotes the sectional curvature and the diameter of $(M,g)$ respectively and $S,D,V>0$ are fixed geometric parameters, is compact in the Gromov-Hausdorff topology. This grants for free uniform bounds on any geometric quantity that is continuous with respect to the Gromov-Hausdorff topology. A drawback of this approach is that the bounds thus obtained are not explicit. Moreover, several interesting quantities are not preserved through Gromov-Hausdorff convergence, like the dimension, the Betti numbers or the eigenvalues of the Laplace-Beltrami operator: some counter-examples are given for instance in \cite{BerardBesson}.

Still, it is tempting to investigate in that way the properties of the set
$$
\mathcal{M}(n,K,D) : = \{ (M^n,g) \, \, \text{closed} \, \, : \, \, \Ric \ge K,\, \,
 \mathrm{diam}\le D\},
$$
where $K\in \setR$ and $D>0$ are fixed, because manifolds in this set satisfy common geometric and analytic properties like the Bonnet-Myers theorem, the Bishop-Gromov theorem, the local $L^2$-Poincaré inequality, and so many more. It turns out that $\mathcal{M}(n,K,D)$ is precompact with respect to the Gromov-Hausdorff topology\footnote{this is Gromov's celebrated precompactness theorem \cite[Th.~5.3]{Gromov}}: from any sequence of Riemannian manifolds belonging to this set, one can extract a convergent subsequence. Limit points obtained in this way are complete metric spaces nowadays called Ricci limit spaces. Their study has produced an abundant literature, with critical works by K.~Fukaya \cite{Fukaya}, J.~Cheeger and T.~Colding \cite{CheegerColdingI,CheegerColdingII,CheegerColdingIII}, T.~Colding and A.~Naber \cite{ColdingNaber}, or S.~Honda \cite{Honda}, to cite only a few.

On the other hand, a well-known result in spectral geometry (see e.g.~\cite[V.~28]{Berard}) asserts that for any $K \in \setR$ and $D>0$, there exists a positive function $Z(t)$ depending only on $n,K,D$ such that
$$
Z_{(M,g)}(t) \le Z(t)
$$
for any $(M,g) \in \mathcal{M}(n,K,D)$ and $t>0$,
where $Z_{(M,g)}(t):=\int_M p(x,x,t) \di \vol(x)=\sum_{i=0}^{+\infty} e^{-\lambda_i t}$ is the trace of the heat kernel of the manifold $(M,g)$. This uniform bound on a spectral quantity, as well as other related spectral results, led to the search for \textit{spectral distances} on the set of closed Riemannian manifolds for which the set $\mathcal{M}(n,K,D)$ would have been at least precompact. Such distances were introduced in \cite{BerardBessonGallot} by means of spectral embeddings.

Let us conclude with mentioning the work of A.~Kasue and H.~Kumura \cite{KasueKumura} in which were introduced other spectral distances via a different heat-kernel based approach. These distances are finer than the Gromov-Hausdorff topology and the sets $\mathcal{M}(n,K,D)$ are precompact for them as well.

\subsection{Spectral embeddings}

Let $(M,g)$ be a closed $n$-dimensional Riemannian manifold with Riemannian volume measure $\vol$ and Laplace-Beltrami operator $\Delta$. Here we follow the convention that $\Delta$ is a non-positive operator. Standard tools from functional analysis show that $-\Delta$ is a densely defined self-adjoint non-negative operator on $L^2(M)$ with discrete spectrum $0=\lambda_0<\lambda_1\le \lambda_2 \le \ldots \to +\infty$, or $0=\nu_0<\nu_1<\nu_2<\ldots \to +\infty$ if we ignore multiplicity, and that $L^2(M)$ can be decomposed into $\oplus_{k=0}^{+\infty} E_k$, where $E_k$ is the space of eigenfunctions associated to the eigenvalue $\nu_k$. Moreover, elliptic regularity theory implies that any eigenfunction of $-\Delta$ is smooth. The operator $-\Delta$ generates a semi-group of self-adjoint operators $(e^{-t\Delta})_{t>0}$ acting on $L^2(M)$. This semi-group admits a positive kernel i.e.~a function $p:M\times M \times (0,+\infty)\to (0,+\infty)$ such that
$$
e^{-t\Delta} f (x) = \int_M p(x,y,t)f(y)\di \vol(y)
$$
holds for any $x \in M$, $t>0$, and $f \in L^2(M)$.
This function $p$, called heat kernel of $(M,g)$, is jointly smooth with respect to any of its three variables. Furthermore, for any orthonormal basis $(\phi_i)_{i \ge 0}$ of $L^2(M)$ adapted to the decomposition $L^2(M) = \oplus_{k=0}^{+\infty} E_k$, the sum of the functions $((t,x,y)\mapsto e^{- \lambda_i t} \varphi_i(x) \varphi_i (y))_{i\ge 0}$ converges in $C^{\infty}((0,+\infty)\times M\times M)$ and provides the so-called spectral decomposition of the heat kernel:
\begin{equation}\label{eq:spectraldecomposition}
p(x,y,t) = \sum_{i = 0}^{+\infty} e^{- \lambda_i t} \varphi_i(x) \varphi_i (y) \qquad \forall x,y \in M,
\end{equation}
see \cite[Th.~10.13]{Grigor'yan}, for instance. Moreover, if $y$ is fixed, we also have
$$
\nabla p(\cdot,y,t) = \sum_{i = 0}^{+\infty} e^{- \lambda_i t} \varphi_i (y) \nabla \varphi_i  \qquad \text{in $L^2(TM)$,}
$$
where $L^2(TM)$ is the Hilbert space of $L^2$ vector fields on $M$ equipped with the scalar product defined by $
\langle V,W\rangle_{L^2(TM)}:=\left(\int_X g(V,W) \di\vol_g \right)^{1/2}$ for any $V,W \in L^2(TM)$.

Now that the appropriate context has been set up, let us introduce the unrescaled and rescaled Bérard-Besson-Gallot spectral embeddings. They both depend on the choice of an orthonormal basis of $L^2(M)$ adapted to the decomposition $L^2(M)=\oplus_{k=0}^{+\infty} E_k$. We write $\mathcal{B}(M,g)$ for the set of such bases.

According to the motivations concerning spectral distances described in the previous section, the most relevant family of embeddings is the following.

\begin{definition}[Unrescaled spectral embeddings]

Let $(M,g)$ be a closed Riemannian manifold. The unrescaled spectral embeddings of $M$ are the maps
$$
I_t^{a}\,  : \, \, \,  
\begin{cases}
\,\, M & \to \quad l^2\\
\,\, x & \mapsto \quad \sqrt{\vol(M)} (e^{-\lambda_i t /2} \phi_i(x))_{i \ge 1}
\end{cases}
$$
where $t>0$ and $a=\{\phi_i\}_{i\ge 0} \in \mathcal{B}(M,g)$.
\end{definition}

\begin{remark}
Note that \eqref{eq:spectraldecomposition} implies $\|I_t^{a}(x)\|_{l^2}^2 = \vol(M)(p(x,x,t)-e^{-\lambda_0 t} \phi_0^2(x))$ and thus $I_t^{a}(x) \in l^2$ for any $x \in M$.
\end{remark}

It is easily checked that the functions $I_t^{a}$ are topological embeddings. Indeed, continuity directly follows from the formula
$$
\|I_t^{a}(x)-I_t^{a}(y)\|_{l^2} = \vol(M)(p(x,x,t)+p(y,y,t)-2p(x,y,t))
$$
which is a straightforward consequence of \eqref{eq:spectraldecomposition}. Injectivity stems from the fact that any basis of $L^2(M)$ separates points. Lastly, since $M$ is compact, any one-to-one continuous function with domain $M$ is necessarily an homeomorphism onto its image.

The functions $I_t^{a}$ permit to define the so-called \textit{spectral distances}.

\begin{definition}[Spectral distances]\label{def:sd}
For any $t>0$ and any closed Riemannian manifolds $(M,g)$ and $(M',g')$, set
\begin{align*}
d_{SD,t}((M,g),(M',g')):=\max &   \left\{
\sup_{a \in \cB(M,g)} \inf_{a' \in \cB(M',g')} \dist_{H,l^2}(I_t^a(M),I_t^{a'}(M')),\right.\\
& \left.\, \, \,  \sup_{a' \in \cB(M',g')} \inf_{a \in \cB(M,g)} \dist_{H,l^2}(I_t^{a'}(M'),I_t^{a}(M))\, \, \, \right\}.
\end{align*}
\end{definition}

Note the analogy with the definition of Hausdorff distance in a metric space $(X,\dist)$: for any $Y,Z\subset X$, the Hausdorff distance between $Y$ and $Z$ is set as $\dist_{H,X}(Y,Z):= \max \left\{\sup_{y \in Y} \inf_{z \in Z} \dist(y,z), \sup_{z \in Z} \inf_{y \in Y} \dist(y,z) \right\}$.

It can be shown that $\dist_{SD,t}$ define distances on the set of isometry classes of closed Riemannian manifolds. Moreover, for any geometric parameters $n, K, D$, the set of isometry classes of $\mathcal{M}(n,K,D)$ is precompact for any of these distances: indeed, suitable estimates on the eigenvalues and eigenfunctions for manifolds in this set show that the embeddings of such manifolds form a subset of the space $h^1:=\{\{\xi_i\}_i \in l^2 \, : \, \sum_{i\ge 1}(1+i^{2/n})|\xi_i|^2 < +\infty \}$ which is precompact in $l^2$ by Rellich's theorem.\\

A rescaled version of the spectral embeddings $I_t^{a}$ was also proposed in \cite{BerardBessonGallot}. These rescaled embeddings turn out to have better properties from a \textit{metric} point of view.

\begin{definition}[Rescaled spectral embeddings]

Let $(M,g)$ be a closed Riemannian manifold. The unrescaled spectral embeddings of $M$ are the maps
$$
\Psi_t^{a} \,  : \, \, \,  
\begin{cases}
\,\, M & \to \quad l^2\\
\,\, x & \mapsto \quad c_n t^{(n+2)/4} (e^{-\lambda_i t /2} \phi_i(x))_{i \ge 1}
\end{cases}
$$
where $t>0$,  $a=(\phi_i)_{i\ge 1} \in \mathcal{B}(M,g)$ and $c_n:=\sqrt{2}(4 \pi)^{n/4}$.
\end{definition}

The $\Psi_t^{a}$ maps are topological embeddings for the same reasons as the $I_t^{a}$ maps are. But they enjoy an additional property: when $t\downarrow 0$, they tend to be an isometry.

\begin{theorem}[Asymptotic isometry]\label{th:asymptoiso}
For any $t>0$ and $a \in \mathcal{B}(M,g)$, the map $\Psi_t^a$ is a smooth embedding. Moreover,
\begin{equation}\label{eq: BBG expansion}
[\Psi_t^a]^*g_{l^2} =  g + \frac{t}{3}\left(\frac{\mathrm{Scal}}{2}\,  g - \mathrm{Ric}\right) + O(t^2) \quad \text{when $t \downarrow 0$},
\end{equation}
in the sense of pointwise convergence, where $\Scal, \Ric$ denotes the scalar  and Ricci curvatures of $(M,g)$ respectively.
\end{theorem}

Here $l^2$ is seen as a Riemannian manifold modelled on a Hilbert space (see e.g.~\cite[1.1, 1.6]{Klingenberg2}) with tangent spaces $T_f l^2$ all canonically isomorphic to $l^2$ itself and equipped with the Riemannian metric 
$$(g_{l^2})_f(\xi,\zeta):=\sum_{i} \xi_i \zeta_i$$
for any $f \in l^2$ and $\xi, \zeta \in T_f l^2$, and $[\Psi_t^a]^*g_{l^2}$ is the pull-back metric defined in the usual way by
$$
([\Psi_t^a]^*g_{l^2})_x(u,v) :=(g_{l^2})_{\Psi_t^a(x)}(D\Psi_t^a(x)\cdot u,D\Psi_t^a(x)\cdot v)
$$
for any $x \in M$ and $u,v \in T_xM$. Using the definitions of all the objects involved, this last line rewrites as:
\begin{equation}\label{eq:17/08}
([\Psi_t^a]^*g_{l^2})_x(u,v) = c_n^2 t^{(n+2)/2}\sum_{i=0}^{+\infty} e^{-\lambda_i t} (D\phi_i(x)\cdot u )(D\phi_i(x)\cdot v).
\end{equation}

\begin{remark}
A version of Theorem \ref{th:asymptoiso} where the metric $g$ depends analytically on the parameter $t$ while the volume form remains constant was proved by H.~Abdallah in \cite{Abdallah}.
\end{remark}

\begin{remark}
In \cite{WangZhu}, X.~Wang and K.~Zhu perturbed the Bérard-Besson-Gallot embeddings to construct a family of \textit{isometric} embeddings $F_t : M^n \to \setR^{q(t)}$ for any $t>0$ small enough, with $q(t) >> t^{-n/2}$.
\end{remark}

\begin{remark}
The equation $$\displaystyle
\frac{\di}{\di t} g_t = \frac{\Scal_{g_t}}{2} g_t - \Ric_{g_t}$$
defines the gradient flow of the Einstein-Hilbert functional $
E : g \mapsto \int_M \Scal \mathrm{d}\vol$ (see \cite[Sect.~6.1]{Topping}, for instance). Therefore, the flow defined by the Bérard-Besson-Gallot spectral embeddings is tangent to the Ricci-Bourguignon flow at the initial time, in the sense that 
$$\displaystyle
\left.\frac{\di}{\di t}\right|_{t=0} g_t = \frac{\Scal}{2} g - \Ric.$$
To the best knowledge of the author, there is no satisfactory explanation of this simple observation in the literature.\\
\end{remark}

Let us sketch a proof of Theorem \ref{th:asymptoiso}. We refer to \cite[Sect.~5.1]{TewodrosePhD} for detailed computations. For convenience, we set $p_t(\cdot,\cdot):=p(\cdot,\cdot,t)$ and $g_t:=[\Psi_t^a]^*g_{l^2}$. Take $x \in M$ and $v \in T_xM$. From \eqref{eq:17/08}, we have
$$
[g_t]_x(v,v) = c_n^2 t^{(n+2)/2} \sum_{i=0}^{+\infty} e^{-\lambda_i t} (D\phi_i(x)\cdot v )^2.
$$
Let us introduce here a useful notation. Let $F:M\times M \to (0,+\infty)$ be a smooth function. We define $\di^1 F (x): T_x M \times M \to (0,+\infty)$ by $\di^1 F(x)(v,y):= (\partial_x F)(x,y) \cdot v$ and $\di^2 [ \di^1 F (x) ](y): T_yM
 \times T_xM \to (0,+\infty)$ by $\di^2 [ \di^1 F (x) ](y) \cdot(v,w)=\partial_y[\di^1F(x)(v,\cdot)](y)\cdot w$ for any $y \in M$, $v \in T_xM$ and $w \in T_yM$. We call mixed derivative of $F$ at $(x,x)$ the map
$$
 D^{\mathrm{mix}} F(x,x) := \di^2 [ \di^1 F (x) ](x): T_xM
 \times T_xM \to (0,+\infty).
$$
A direct computation with the spectral decomposition \eqref{eq:spectraldecomposition} shows that
 $$
  D^{\mathrm{mix}} p_t(x,x) \cdot (v,v) = \sum_{i=0}^{+\infty} e^{-\lambda_i t} (D\phi_i(x)\cdot v )^2
$$
hence
\begin{equation}\label{eq:17/08.2}
 [g_t]_x(v,v) = c_n^2 t^{(n+2)/2} D^{\mathrm{mix}} p_t(x,x) \cdot (v,v).
\end{equation}
Let us recall a classical result which is a consequence of the celebrated Minakshisundaram-Pleijel small-time expansion of the heat kernel \cite{MinakshisundaramPleijel, BergerGauduchonMazet}. We write $\mathrm{inj}(M)$ for the injectivity radius of $(M,g)$ and use the classical notation $\exp_x$ for the exponential map at $x$. Recall that $\exp_x$ is a smooth diffeomorphism from $B_{\mathrm{inj}(M)}(0_n) \subset \setR^n \simeq T_xM $ to $B_{\mathrm{inj}(M)}(x)$, where $0_n$ denotes the origin in $\setR^n$.

\begin{claim}
For any $x,y \in M$ such that $\dist(x,y) < \mathrm{inj}(M)$, where $\dist$ is the canonical Riemannian distance on $(M,g)$,
\begin{equation}\label{eq:MP}
p(x,y,t) = \frac{1}{(4\pi t)^{n/2}} e^{-\frac{\dist^2(x,y)}{4t}}(u_0(x,y) + t u_1(x,y) + O(t^2))
\end{equation}
when $t \downarrow 0$, where $u_0(x,y)=\theta_x(\exp_x^{-1}(y))^{-1/2}$, $\theta_x$ is the density of $(\exp_x^{-1})_\#\vol$ with respect to the Lebesgue measure in $T_xM \simeq \setR^n$, and $u_1$ is a smooth function such that $u_1(x,x)=\Scal(x)/6$. Moreover, \eqref{eq:MP} can be differentiated as many times as desired with respect to any of the spatial variables $x$ and $y$.
\end{claim}

Thus
\begin{align*}
D^{\mathrm{mix}} p_t(x,x) \cdot (v,v) & = \frac{1}{(4 \pi t)^{n/2}} \left[ D^{\mathrm{mix}} e^{-\frac{\dist^2(x,x)}{4t}} \cdot (v,v) \bigg( u_0(x,x)+tu_1(x,x)+O(t^2)\bigg) \right.\\
& \qquad \qquad \qquad \qquad \qquad \left. + \, \underbrace{e^{-\frac{\dist^2(x,x)}{4t}}}_{=1}\bigg( D^{\mathrm{mix}} u_0(x,x) \cdot (v,v) + O(t)\bigg) \right].
\end{align*}
A direct computation shows that
$$
D^{\mathrm{mix}} e^{-\frac{\dist^2(x,x)}{4t}} \cdot (v,v) =  -\frac{1}{4t} D^{\mathrm{mix}} \dist^2(x,x) \cdot (v,v) =  \frac{g_x(v,v)}{2t}\, \cdot
$$
Moreover, $u_0(x,x)=\theta_x(0)^{-1/2}=1$, and as well-known, if $(r,u)$ are polar coordinates on $T_xM$, then
\begin{equation}\label{eq:expRicci}
\theta_x(r,u)=r^{n-1}(1-(r^2/6)\Ric_x(u,u)+O(r^3))
\end{equation} when $r \downarrow 0$, what leads to
$$
D^{\mathrm{mix}} u_0(x,x) \cdot (v,v) = - \frac{1}{6}\Ric_x(v,v)
$$
via another direct computation. Thus
$$
D^{\mathrm{mix}} p_t(x,x) \cdot (v,v) = \frac{1}{(4 \pi t)^{n/2}} \left[ \frac{g_x(v,v)}{2t}\left( 1 + t \frac{\Scal(x)}{6}\right)-\frac{1}{6}\Ric_x(v,v) + O(t)\right].
$$
The result follows from multiplying this last line by $2t(4\pi t)^{n/2}=c_n^2 t^{(n+2)/2}$.

\subsection{Truncated versions}

For numerical applications (see Section 3), it is important to know whether truncating the rescaled spectral embeddings $\Psi_t^{a}$ keep their nice metric properties. Here by truncating we mean to consider the maps
\begin{equation}\label{eq:truncatedembedding}
[\Psi_t^{a}]^m\,  : \, \, \,  
\begin{cases}
\,\, M & \to \quad \setR^m\\
\,\, x & \mapsto \quad c_n t^{(n+2)/4} (e^{-\lambda_1 t /2} \phi_1(x),\ldots,e^{-\lambda_m t /2} \phi_m(x))
\end{cases}
\end{equation}
where $m$ is a positive integer. The main question is then:\\

does there exist an integer $m$ depending only on the dimension $n$ and on geometric parameters $a_1,a_2,\ldots$ such that for any closed $n$-dimensional Riemannian manifold with geometric constraints dictated by $a_1,a_2,\ldots$, the maps $[\Psi_t^{a}]^m$ are smooth embeddings for any $t>0$ and close to being an isometry for $t$ smaller than a uniform threshold $t_o>0$?\\

In fact, this question is not without interest neither from a theoretical point, especially when placed in perspective with Whitney's and Nash's celebrated embedding theorems.

In \cite{Bates}, J.~Bates studied this question in the case of the functions
$$
\Phi^{a}\,  : \, \, \,  
\begin{cases}
\,\, M & \to \quad \setR^\setN\\
\,\, x & \mapsto \quad (\phi_i(x))_{i\ge 1}.
\end{cases}
$$
He called \textit{maximal embedding dimension} of $M$ the smallest positive integer $m$ such that the map
$$
[\Phi^{a}]^m\,  : \, \, \,  
\begin{cases}
\,\, M & \to \quad \setR^m\\
\,\, x & \mapsto \quad (\phi_1(x),\ldots,\phi_m(x)).
\end{cases}
$$
is a smooth embedding for any $a=\{\phi_i\}_i \in \mathcal{B}(M,g)$. Note that the map $[\Phi^{a}]^m$ is a smooth embedding if and only if for some (or any) $t>0$ the  map $[\Psi_t^{a}]^m$ is so too, as the two maps are equal up to composition by the invertible matrix $c_n t^{(n+4)/2}\mathrm{Diag}(e^{-\lambda_1 t}, \ldots, e^{-\lambda_n t})$.

Bates' main result is the existence of a uniform upper bound on the maximal embedding dimension of Riemannian manifolds satisfying suitable geometric constraints.

\begin{theorem}
Let $K_o\ge 0$, $i_o>0$ be fixed constants and $n\ge 2$ an integer. Then there exists $\overline{m}=\overline{m}(K_o,i_o,n) \in \setN$ such that $[\Phi^{a}]^{\overline{m}} : M \to \setR^{\overline{m}}$ is a smooth embedding for any $a \in \mathcal{B}(M,g)$, where $(M,g)$ is any closed $n$-dimensional Riemannian manifold satisfying
\begin{equation}\label{eq:geoconstr}
\Ric \ge - (n-1) K_o g, \qquad \mathrm{inj}(M) \ge i_o, \qquad \vol(M)=1.
\end{equation}
\end{theorem}
The proof can be divided into two steps:
\begin{itemize}
\item[(a)] to show the existence of a dimension $\overline{m}=\overline{m}(K_o,i_o,n)$ and a radius $\overline{r}=\overline{r}(K_o,i_o,n)>0$ such that the map $[\Phi^a]^{\overline{m}}$ is a smooth embedding when restricted to any ball $B_r$ in $M$ with $r \in (0,\overline{r})$,

\item[(b)] to check that any two points $x,y \in M$ such that $\dist(x,y)>\overline{r}$ are distinguished by $[\Phi^a]^{\overline{m}}$.
\end{itemize}

Roughly speaking, (a) is achieved via the combination of three results. We need to introduce some terminology and notation to state them. We call \textit{normalized} any coordinate patch $(U,h)$ around a point $x$ of a Riemannian manifold $(M,g)$ such that $h(x)=0_n$ and $g_{ij}(0_n)=\delta_{ij}$ for any $i,j \in \{1,\ldots,n\}$, where $(g_{ij})_{i,j}$ are the coefficients of $g$ read in $h$. We write $G$ for the $n \times n$ matrix $[g_{ij}]_{i,j}$. Note that $g_{ij} \in C^{\infty}(h(U))$ for any $i,j$ and that we can define a family of norms on $\setR^n$, parametrized by $h(U)$, by setting $|\xi|_{G(\cdot)}:=\sum_{i,j} g_{i,j}(\cdot)\xi_i \xi_j$ for any $\xi \in \setR^n$. For brevity, we write $|\xi|_{G}$ instead of $|\xi|_{G(\cdot)}$. Lastly, we recall that for any $\alpha \in (0,1)$, the $C^\alpha$-norm of a function $f: \Omega \to \setR$, where $\Omega \subset \setR^n$ is an open set, is set as
$$
[f]_\alpha:= \sup_{x \in \Omega} |f(x)| + \sup_{x,y \in \Omega, x\neq y} \frac{|f(x)-f(y)|}{\|x-y\|_2^\alpha}\, \cdot
$$

The first result comes from the work of P.~Jones, M.~Maggioni and R.~Schul \cite{JonesMaggioniSchul}. It asserts that if $(M^n,g)$ is a closed Riemannian manifold with $\vol(M)=1$, if there exists a normalized coordinate patch $(U,h)$ around a point $z\in M$ satisfying
\begin{itemize}
\item[(i)] $h(U)=B_r(0_n)$ for some $r>0$,
\item[(ii)] $Q^{-1}\|\cdot\|_2\le |\cdot|_{G} \le Q \|\cdot\|_2$ for some $Q>1$,
\item[(iii)] $\sup_{i,j} [g_{ij}]_{C^\alpha(B_r(0_n))}\le C$ for some $\alpha \in (0,1)$ and $0<C<+\infty$,
\end{itemize}
 then there exist a constant $\kappa=\kappa(n,Q,\alpha,C)>1$ and positive integers $i_1,\ldots,i_n$ such that for any $a=\{\phi_i\} \in \mathcal{B}(M,g)$, the map
 $$
\tilde{\Phi}^a\,  : \, \, \,  
\begin{cases}
\,\, B_{\kappa^{-1}r}(z) & \to \quad \setR^n\\
\quad \quad x & \mapsto \quad (\gamma_1 \phi_{i_1}(x),\ldots,\gamma_n \phi_{i_n}(x)), 
\end{cases}
$$
where $\gamma_j := (\fint_{B_{\kappa^{-1}r} }\phi_{i_j})^{-1/2}$ for any $j \in \{1,\ldots,n\}$, satisfies
$$
\frac{\kappa^{-1}}{r} \dist(x,y) \le \|\tilde{\Phi}^a(x)-\tilde{\Phi}^a(y)\|_2\le \frac{\kappa}{r} \dist(x,y)
$$
for any $x,y \in B_{\kappa^{-1}r}(z)$, where $\dist$ is the canonical Riemannian distance of $(M,g)$; in particular, $\tilde{\Phi}^a$ is a bi-Lipschitz embedding onto its image. Moreover, the eigenvalues associated with the eigenfunctions $\phi_{i_1}, \ldots, \phi_{i_n}$ satisfy:
\begin{equation}\label{eq:moreover}
\frac{\kappa^{-1}}{r^2} \le \lambda_{i_j} \le \frac{\kappa}{r^2} \qquad \forall j \in  \{1,\ldots,n\}.
\end{equation}
Of course since the $\phi_i$ are smooth, then $\tilde{\Phi}^a$ is smooth too.

The second result is the construction of $C^\alpha$-harmonic coordinate patches by M.~Anderson and J.~Cheeger \cite{AndersonCheeger}. Let us recall that a coordinate patch $h=(h_1,\ldots,h_n)$ on an $n$-dimensional Riemannian manifold $(M,g)$ is called \textit{harmonic} whenever $\Delta h_{i}=0$ for any $i$. Moreover, for $\alpha \in (0,1)$ given, $h$ is called $C^\alpha$-harmonic if it is harmonic, normalized, satisfying (i) and (ii) given above and:
\begin{itemize}
\item[(iii)'] $r^\alpha \sup_{i,j} [g_{ij}]_{C^\alpha(B_r(0_n))}\le Q-1$.
\end{itemize}
M.~Anderson and J.~Cheeger proved that any $n$-dimensional Riemannian manifold satisfying the constraints \eqref{eq:geoconstr} can be covered by an atlas of $C^\alpha$-harmonic coordinate patches with a uniform radius called \textit{harmonic radius} that depends only on $n, K_o, i_o, \alpha, Q$. Bates adapted the proof of this result to show that the same result holds with (iii)' replaced by (iii). This permits to apply the previous result.

The third result is a uniform bound on the eigenvalues of the Laplace-Beltrami operator for closed Riemannian manifolds satisfying suitable geometric constraints, see \cite[Th.~3]{BerardBessonGallot}: for any $D>0$, there exists a constant $C_{s}=C_{s}(n,K,D)>0$ such that for any closed connected $n$-dimensional Riemannian manifold $(M,g)$ with $\Ric\ge -(n-1)K g$ and $\mathrm{\diam}(M)\le D$, one has:
\begin{equation}\label{eq:BBGbounds}
\lambda_i \ge C_{s} i^{2/n} \qquad \forall{i \ge 0}.
\end{equation}

Since there exists $D>0$ such that $\diam(M)\le D$ for any $(M,g)$ such that $\vol(M)=1$, then \eqref{eq:BBGbounds} is in force when we assume the constraints \eqref{eq:geoconstr}. Applying \eqref{eq:moreover}, we get
$$
\frac{\kappa}{r^2} \ge \lambda_{i_j} \ge C_{s} i_j^{2/n} \qquad \forall j \in \{1,\ldots,n\}.
$$
Choosing $\overline{m}$ as the ceiling value of $(\kappa r^{-2}C_{s}^{-1})^{n/2}-1$ ensures $C_{s} (\overline{m}+1)^{2/n} > \kappa r^{-2}$ and thus prevents $i_1, \ldots, i_n$ to be greater than $\overline{m}$. Then $[\Phi^{a}]^{\overline{m}}$ is a local embedding because it takes into account the maps $\phi_{i_1}, \ldots, \phi_{i_n}$ that separate points in balls of $M$.

\hfill

The proof of (b) is obtained by a quick computation using suitable information on the heat kernel. We do not provide the details here, see \cite[Sect.~3]{Bates}.

\subsection{Quantitative versions}

If $(M,g)$ is a smooth Riemannian manifold and $F:M\to \setR^N$ a smooth function, the norm of the differential of $F$ at a point $x\in M$, namely $$|DF(x)|:=\sup\{\|DF(x)\cdot v\| \, : \, v \in T_xM \, \,\text{with} \, \, g(v,v)=1\},$$ provides an estimate on how far the function $F$ is to be a local isometry in a neighborhood of $x$. Of course $|DF(x)|$ depends on the norm $\|\cdot\|$ we put on $\setR^N$. Here we will only consider the Euclidean norm $\|\cdot\|_2$ and the infinity norm $\|\cdot\|_{\infty}$ and write $|D_x F|_2$ and $|D_x F|_{\infty}$ respectively for the corresponding norms of $F$ at $x$ measured with these two norms.

It follows from Theorem \ref{th:asymptoiso} that the Bérard-Besson-Gallot rescaled spectral embeddings $(\Psi^a_t)_{t>0}$ are such that $$|D \Psi^a_t(x)|_2 \to 1$$ when $t\to 0^+$. This observation raises a natural question: for any accuracy parameter $\eps>0$, does there exist a threshold time $t_o>0$ such that for any $t\in (0,t_o)$, $$1-\eps < |D \Psi^a_t|_2 < 1+\eps \, \,?$$
Here and in the sequel we omit $x$ in $|D\Psi^a_t(x)|$ to mean that the statement holds for any $x \in M$.

In \cite{Portegies}, J.~Portegies answered this question for several families of spectral embeddings including the Bérard-Besson-Gallot rescaled ones, by considering the sets of closed Riemannian manifolds
$$\mathcal{M}(n,K,i,V):=\{ (M^n,g) \, \, : \, \, \Ric \ge Kg,\, \,
 \mathrm{inj}(M)\ge i, \, \, \vol(M)\le V\},
$$
where $n \in \setN\backslash\{0\}$,  $K \in \setR$ and $i ,V>0$. From now on we consider these parameters as fixed.

Let us begin with a family of maps introduced by Portegies himself. Let $(M,g)$ be a closed $n$-dimensional Riemannian manifold. For any positive integer $N$, any $q_1, \ldots, q_N \in M$ and any $t>0$, consider the smooth map 
$$
G_{(q_1,\ldots,q_N),t}\,  : \, \, \,  
\begin{cases}
\,\, M & \to \quad \setR^N\\
\,\, \, x & \mapsto \quad c_nt^{(n+1)/2}(p(q_1,x,t),\ldots,p(q_N,x,t))
\end{cases}
$$
where $c_n=\sqrt{2}\pi^{-n/2}e^{-1/2}$ is a scaling factor. One can understand this map as giving a picture of $M$ snapshotted by the heat kernel at time $t$ from the viewing points $q_1,\ldots, q_N$. Therefore, the more viewing points we have at our disposal, the better we might recover the manifold. In order to catch the geometry of the manifold in an optimal way, it sounds natural to consider points in a $\delta$-net\footnote{we recall that for any $\delta>0$, a $\delta$-net of a metric space $(X,\dist)$ is a subset $Y$ of $X$ such that the balls $\{B_{\delta/2}(y)\}_{y \in Y}$ are disjoint while the union of the balls $\{B_{\delta}(y)\}_{y \in Y}$ covers $X$}, with $\delta$ small.  From a computational perspective though, we cannot afford to use too many points, so $\delta$ cannot be too small. The next theorem gives a quantitative estimate on $\delta$ to ensure that the map $G_{(q_1,\ldots,q_N),t}$ associated with any $\delta$-net $(q_1,\ldots,q_N)$ is in a range $\eps$ from being an isometric embedding of $M$ into $\setR^N$ equipped with the norm $\|\cdot\|_\infty$.

\begin{theorem}[Net heat kernel embeddings in $(\setR^N,\|\cdot\|_\infty)$]\label{th:Portegies1}

For any $\eps>0$, there exists $t_o>0$ depending only on $n,K,i, \eps,$ and $N_o \in \setN$ depending only on $n,K,i,V,\eps$ such that for any $t\in(0,t_o)$, there exists a net parameter $\delta>0$ depending only on $n,K,i,\eps,t$ such that for any $\delta$-net $\{q_1,\ldots,q_{N_o}\}$ of $(M,g) \in \mathcal{M}(n,K,i,V)$, the map $G_{(q_1,\ldots,q_{N_o}),t}$ is an embedding satisfying $$1-\eps < |DG_{(q_1,\ldots,q_{N_o}),t}|_{\infty}< 1+\eps.$$
\end{theorem}

\begin{remark}
Portegies' net heat kernel embeddings are inspired by Gromov's variant of the Kuratowski distance functions embedding \cite{Gromov}. The link between heat kernel embeddings and distance functions embeddings is given by Varadhan's celebrated formula:
$$
\lim\limits_{t \to 0} - 4t \log(p(x,y,t)) = \dist^2(x,y).
$$
\end{remark}

\begin{remark}
Portegies also provided a \textit{truncated} version of Theorem \ref{th:Portegies1}, namely where the heat kernel involved in the definition of $G$ is replaced by $p^{N}$ which is defined by keeping only the $N$-th first terms in the spectral decomposition \eqref{eq:spectraldecomposition}.
\end{remark}

In order to replace the uniform norm on $\setR^N$ by the Euclidean norm, Portegies introduced a family of suitably weighted heat kernel embeddings. This leads to the following embedding theorem which states, roughly speaking, that for any $\eps>0$, one can select a finite number of points on a manifold in $\mathcal{M}(n,K,i,V)$ from which to build a heat kernel embedding that is $\eps$-close to be an isometry. We call \textit{special points} these heat kernel embeddings.

\begin{theorem}[Special points heat kernel embeddings into $(\setR^N,\|\cdot\|_2)$]
For any $\eps>0$, there exists $t_o>0$ depending only on $n,K,i, \eps$ such that for any $t\in(0,t_o)$, there exists $\lambda>0$ and $N \in \setN$ both depending only on $n,K,i,V,\eps,t$ such that for any $(M,g) \in \mathcal{M}(n,K,i,V)$, there exists points $\overline{q}_1,\ldots,\overline{q}_{N} \in M$ such that the map
$$
H^{t,\lambda}_{\overline{q}_1,\ldots,\overline{q}_{N}}\,  : \, \, \,  
\begin{cases}
\,\, M & \to \quad \setR^{N}\\
\,\, x & \mapsto \quad c_n't^{\frac{n+2}{4}}\lambda(p(\overline{q}_1,x,t),\ldots,p(\overline{q}_{N},x,t)),
\end{cases}
$$
where $c_n':=2^{(3n+4)/4}\pi^{n/4}$ is a scaling factor, is an embedding satisfying
$$
1-\eps \le |DH^{t,\lambda}_{\overline{q}_1,\ldots,\overline{q}_{N}}|_2 \le 1+\eps.
$$
\end{theorem}

Finally, let us provide Portegies' truncated and quantitative version of the Bérard-Besson-Gallot almost-isometry theorem.

\begin{theorem}[Almost isometric truncated Bérard-Besson-Gallot spectral embeddings]\label{th:truncated}
For any $\eps>0$, there exists $t_o>0$ depending only on $n,K,i,\eps$ such that for any $t \in (0,t_o)$, there exists a truncation number $\overline{m}$ depending only on $n,K,i,V,\eps, t$ such that for any $m \ge \overline{m}$, for any $(M,g) \in \mathcal{M}(n,K,i,V)$ and any $a\in \mathcal{B}(M,g)$, the map $[\Psi_t^a]^m:M \to \setR^N$ is a smooth embedding satisfying $$1-\eps < |D [\Psi_t^a]^m|_2 < 1+\eps.$$
\end{theorem}

The proof of all these results is based on a quantitative construction on the harmonic radius in terms of the parameters $n,K,i,V$. The proof by M.~Anderson and J.~Cheeger was based on a contradiction argument which did not provide such a quantitative estimate. By suitably exploiting the Bishop-Gromov theorem and the segment inequality, one can show that for any $x \in M \in \mathcal{M}(n,K,i,V)$ and any orthonormal basis $\{e_i\}_i$ of $T_xM$, the functions $\{\dist(p_i,\cdot)\}_i$, where $p_i:=\exp_x(e_i/4)$ for any $i\in \{1,\ldots,n\}$, form a coordinate patch satisfying (i), (ii) and (iii)' with domain $B_{\overline{r}}(x)$ where $\overline{r}>0$ depends only on $n,K,i,V$. To turn these coordinates into harmonic ones, one can replace them by their so-called \textit{harmonic replacement}, that is the solution of the Dirichlet problem
$$
\begin{cases}
\Delta h_i = 0 & \text{on $B_{\overline{r}}(x)$}\\
h_i(\cdot)=\dist(p_i,\cdot) & \text{on $\partial B_{\overline{r}}(x)$}.
\end{cases}
$$
Bu using suitable interior elliptic estimates, a uniform bound on the distance functions for manifolds in $\mathcal{M}(n,K,i,V)$, and a quantitative version of the maximum principle, it can be shown that the maps $\{h_i\}_i$ form a harmonic coordinate patch with domain $B_{\overline{r}}(x)$. We refer to \cite[Appendix]{Portegies} for more details.

\subsection{Heat kernel embeddings via the connection Laplacian}

In \cite{SingerWu}, A.~Singer and H.-T.~Wu introduced another family of unrescaled spectral embeddings
$$
V_t^a : M \to l^2
$$
of a closed Riemannian manifold $(M,g)$. Their approach, very similar to the one adopted in \cite{BerardBessonGallot}, relied on the heat kernel of the \textit{connection Laplacian} for vector fields. Let us recall that the connection Laplacian for vector fields on $M$ is the Friedrich extension $\Delta_C : L^2(TM) \to L^2(TM)$ of the symmetric operator $\Delta_C : C_c^\infty(TM) \to C_c^\infty(TM)$ defined by
$$
\Delta_C V := \mathrm{tr}(\nabla^2 V)
$$
for any $V \in C^\infty(TM)$,
where $\nabla^2 V:C^\infty(TM)\times C^\infty(TM) \to C^\infty(TM)$ is the second covariant derivative of $V$, $\mathrm{tr}$ is the trace operator on $(1,2)$ tensors and $C_c^\infty(TM)$ (resp.~$C^\infty(TM)$) is the space of compactly supported smooth (resp.~smooth) vector fields on $M$.

The operator $-\Delta_C$ is a non-negative self-adjoint second-order elliptic operator on $L^2(TM)$ admitting a discrete spectrum $0=\lambda_{0}^C < \lambda_{1}^C \le \lambda_{2}^C \le \ldots \to +\infty$, counted with multiplicity. Moreover, if we denote by $0<\nu_1^C<\nu_2^C<\ldots\to +\infty$ the same spectrum counted without multiplicity, we can decompose $L^2(TM)$ into $\oplus_{k=1}^{+\infty}E_k$, where $E_k$ is the eigenspace of $-\Delta_C$ corresponding to the eigenvalue $\nu_k^C$. We denote by $\cB_C(E_k)$ the set of orthonormal bases of $E_k$, and define $
\cB_C(M,g):=\prod_{k=1}^{+\infty} \cB_C(E_k)$.

The semigroup $(e^{-t\Delta_C})_{t>0}$ generated by $-\Delta_C$ admits a kernel $p_C$ such that for any $x,y \in M$ and $t>0$,
$$p_C(x,y,t):T_yM \to T_xM,$$
i.e.
$$
e^{-t\Delta_C} V(x) = \int_M p_C(x,y,t) V(y)\di \vol(y)$$
for any $x \in M$, $V \in L^2(TM)$ and $t>0$. Moreover, $p_C$ is smooth in $x$ and $y$ and analytic in $t$. Finally, just like the classical heat kernel, $p_C$ admits a spectral decomposition: for any $x,y \in X$ and $t>0$,
\begin{equation}\label{eq:spectralconnection}
p_C(x,y,t) = \sum_{i=0}^{+\infty} e^{-\lambda_i^C t} X_i(x) \otimes X_i^{\sharp}(y)
\end{equation}
holds, where $(X_i)_{i \ge 0} \in \mathcal{B}_C(M,g)$ and $\cdot^{\sharp}$ is the sharp operator on vector fields.

\begin{definition}[Vector Diffusion Maps]\label{def:VDM}

Let $(M,g)$ be a closed Riemannian manifold. The vector diffusion maps of $M$ are the functions 
$$
V_t^{a}\,  : \, \, \,  
\begin{cases}
\,\, M & \to \quad l^2\\
\,\, x & \mapsto \quad \vol(M) (e^{-(\lambda_i^C + \lambda_j^C)t/2}g_x(X_i(x),X_j(x)))_{i,j\ge1}
\end{cases}
$$
where $t>0$ and $a=(X_i)_{i\ge 0} \in \mathcal{B}_C(M,g)$.
\end{definition}

Using the spectral decomposition \eqref{eq:spectralconnection}, one can easily prove that the vector diffusion maps are smooth embeddings and that the quantity
$$
\|V_t^{a}(x)-V_t^{a}(y)\| 
$$
does not depend on the choice of $a \in \cB(M,g)$ for any $x,y \in M$ and $t>0$. Calling it $$\dist_{VDM,t}(x,y),$$ we get distances $\dist_{VDM,t}$ on $M$ called \textit{vector diffusion distances}. The next theorem (\cite[Th.~8.2]{SingerWu}) provides an information on how $\dist_{VDM,t}$ behaves when $t \downarrow 0$:

\begin{theorem}[Short-time behavior of the Vector Diffusion Distances]\label{th:short-time}
Let $(M,g)$ be a closed $n$-dimensional Riemannian manifold. Then for any $x\in M$, there exists $t_o>0$ and $C>0$ such that for all $0<t<t_o$ and all $v \in T_xM$ with $\|v\|^2 \ll t$, setting $y=\exp_x(v)$ one has
$$
|c_nt^{n+1}\dist^2_{VDM,t}(x,y)- \|v\|^2| \le Ct,
$$
with $c_n:=(4\pi)^{n}/n$.
\end{theorem}

The proof of this theorem relies on the analogue of the Minakshisundaram-Pleijel expansion for the connection heat kernel (see e.g.~\cite{BerlineGetzlerVergne}): for any $x \in M$ and $v \in T_{x}M$ with $\|v\| \le t \le \mathrm{inj}(M)$, when $t \downarrow 0$, writing $y$ for $\exp_x(v)$ we have
\begin{equation}\label{eq:MP2}
p_C(x,y,t) = \frac{1}{(4\pi t)^{n/2}} e^{-\frac{\|v\|^2}{4t}}\theta_x(\exp_x^{-1}(y))^{-1/2}(\Phi_0(x,y) + t \Phi_1(x,y) + O(t^2)),
\end{equation}
where $\Phi_0(x,y)$ is the parallel transport from $T_xM$ to $T_yM$. Since
\begin{align*}
\dist^2_{VDM,t}(x,y) & = \mathrm{tr}(p_C(x,x,t)p_C(x,x,t)^\sharp) + \mathrm{tr}(p_C(y,y,t)p_C(y,y,t)^\sharp)\\
& -  2\mathrm{tr}(p_C(x,y,t)p_C(x,y,t)^\sharp),
\end{align*}
applying \eqref{eq:MP2} and \eqref{eq:expRicci} yields to the result, thanks to a simple computation.\\

The truncated vector diffusion maps were subsequently studied by C.-Y.~Li and H.-T.~Wu who proved in \cite{LinWu} that for any closed smooth Riemannian manifold $(M,g)$ and any $t>0$, there exists a positive integer $\overline{m}$ such that for any $m \ge \overline{m}$ and any $a \in \mathcal{B}(M,g)$, the map
\[
[V_t^{a}]^{m^2}\,  : \, \, \,  
\begin{cases}
\,\, M & \to \quad \setR^{m^2}\\
\,\, x & \mapsto \quad \vol(M) (e^{-(\lambda_i^C + \lambda_j^C)t/2}g_x(X_i(x),X_j(x)))_{1\le i,j \le m}
\end{cases}
\]
is a smooth embedding. However, a quantitative stamement in the form of Theorem \ref{th:truncated} is still open:\\

\textbf{Open problem 1.}
Let $\mathcal{M}(\lambda_1,\ldots,\lambda_k)$ be a class of (isometry classes of) closed Riemannian manifolds depending on some geometric parameters $\lambda_1, \ldots, \lambda_k$. For any small $\eps>0$, does there exist $t_o>0$ depending only on $\lambda_1, \ldots, \lambda_k, \eps$ such that for any $t \in (0,t_o)$, there exists a truncation number $\overline{m}$ depending only on $\lambda_1, \ldots, \lambda_k,\eps, t$ such that for any $m \ge \overline{m}$, for any $(M,g) \in \mathcal{M}(\lambda_1, \ldots, \lambda_k)$ and any $a\in \mathcal{B}(M,g)$, the map $[V_t^{a}]^{m^2}$ is a smooth embedding satisfying $$1-\eps < |D [V_t^{a}]^{m^2}|_2 < 1+\eps.$$

\hfill

In \cite{Wu}, H.-T.~Wu used the vector diffusion maps to define the so-called \textit{vector spectral distances}.

\begin{definition}[Vector Spectral Distances]
For any $t>0$ and any closed connected smooth Riemannian manifolds $(M,g)$ and $(M',g')$, set
\begin{align*}
\dist_{VDM,t}((M,g),(M',g')):=\max & \left\{
\sup_{a \in \cB(M,g)} \inf_{a' \in \cB(M',g')} \dist_{H,l^2}(V_t^a(M),V_t^{a'}(M')),\right.\\
& \left.\, \, \,  \sup_{a' \in \cB(M',g')} \inf_{a \in \cB(M,g)} \dist_{H,l^2}(V_t^{a'}(M'),V_t^{a}(M))\, \, \, \right\}.
\end{align*}
\end{definition}
\noindent He showed that $\dist_{VDM,t}$ defines a distance on the set of isometry classes of closed Riemanian manifolds for any $t>0$ for which the sets $\mathcal{M}(n,K,D)$ are precompact. A natural question raised by the referee is the following:\\

\textbf{Open problem 2:} Hoes does the vector spectral distances relate to the spectral distances of Definition \ref{def:sd}?\\

As a reasonable first step to answer this question, one could study the relationship between the eigenvalues/eigenfunctions of the Laplace-Beltrami operator with the eigenvalues/eigenvector fields of the connection Laplacian. A result in that direction was found by B.~Colbois and D.~Maerten \cite{ColboisMaerten}: for any closed connected Riemannian manifold $(M^n,g)$ with $\Ric \ge -K(n-1)$ for some $K\ge0$,
\[
\lambda_i^C \le K(n-1) + \lambda_i
\]
for any $i$. However, to the best knowledge of the author, nothing else has been obtained so far.

\subsection{Heat kernel embeddings for possibly non-smooth spaces}

In the recent \cite{AmbrosioHondaPortegiesTewodrose}, the Bérard-Besson-Gallot spectral embeddings have been studied in the context of compact metric measure spaces satisfying the synthetic Riemannian Curvature-Dimension condition $\RCD(K,N)$, where $K \in \setR$ and $N \ge 1$ must be understood as a lower bound on the Ricci curvature and an upper bound on the dimension, respectively. This condition has been under extensive study over the past few years: see for instance the survey \cite{Ambrosio}. We provide here a brief introduction and refer to \cite{TewodrosePhD} for details and references.\\

\textbf{A brief introduction to $\RCD(K,N)$ spaces}

Let $(X,\dist,\meas)$ be a Polish metric measure space, meaning here a Polish (i.e.~complete and separable) metric space equipped with a fully supported Borel regular measure $\meas$ that is finite and non-zero on balls with finite and non-zero radius. Note that the assumption ``$\meas$ fully supported'' can be removed to the prize of technical complications we do not want to enter to here.

 The Cheeger energy of $(X,\dist,\meas)$ is the functional defined on $L^2(X,\meas)$ by setting
\begin{equation}\label{eq:relax2}
\Ch(f)=\inf_{f_n \to f} \left\{ \liminf\limits_{n \to +\infty} \int_X |\nabla f_n|^2 \di \meas \right\} \in [0,+\infty]
\end{equation}
for any $f \in L^2(X,\meas)$, where the infimum is taken over the set of sequences $\{f_n\}_n \subset L^2(X,\meas) \cap \Lip(X)$ such that $\|f_n - f\|_{L^2(X,\meas)} \to 0$, and where for any locally Lipschitz function $F$, the function $|\nabla F|$ -- called slope of $F$ -- is set as
$$
|\nabla F|(x) :=
\begin{cases}
\limsup\limits_{y \to x} \frac{|F(x)-F(y)|}{\dist(x,y)} &  \text{if $x \in X$ is not isolated},\\
\qquad \quad 0 & \text{otherwise}.
\end{cases}
$$
The Cheeger energy must be understood as an abstract extension of the classical Dirichlet energy of $\setR^n$ defined on the Sobolev space $H^1$ by $$\mathrm{Di}(f):= \int_{\setR^n} |\nabla f|^2$$ for any $f \in H^1$. Accordingly, the finiteness domain of $\Ch$ is called (Cheeger metric measure) Sobolev space of $(X,\dist,\meas)$ and usually denoted $$H^{1,2}(X,\dist,\meas).$$

A suitable diagonal argument applied to optimal approximating sequences in \eqref{eq:relax2} provides for any $f \in H^{1,2}(X,\dist,\meas)$ the existence of an $L^2$-function $|\nabla f|_{*}$, called \textit{minimal relaxed slope} of $f$, which gives integral representation of $\Ch$, that is:
$$
\Ch(f)=\int_X |\nabla f|_{*}^2 \di \meas.
$$
The minimal relaxed slope is a local object, meaning that $
|\nabla f|_* = |\nabla g|_*$ $\meas$-a.e.~on $\{f = g\}$ for any $f, g \in H^{1,2}(X,\dist,\meas)$. This, combined with the integral representation property, ensures that $|\nabla f|_*$ is unique as a class of $L^2$-equivalent functions.

When equipped with the norm $\|\cdot\|_{H^{1,2}}:=(\|\cdot\|_{L^{2}}^2 + \Ch(\cdot))^{1/2}$, the space $H^{1,2}(X,\dist,\meas)$ is always a Banach space, but it might fail to be a Hilbert space: this is the case for instance when $(X,\dist,\meas)$ is a smooth non-Riemannian Finsler manifold. In case $H^{1,2}(X,\dist,\meas)$ is a Hilbert space, we say that $$
\text{$(X,\dist,\meas)$ is infinitesimally Hilbertian.}
$$A smooth Riemannian manifold equipped with the canonical Riemannian distance and volume measure is an obvious case of an infinitesimally Hilbertian metric measure space; actually, as shown in \cite{LucicPasqualetto}, a Riemannian manifold equipped with the Riemannian distance and any arbitrary positive Radon measure is always infinitesimally Hilbertian.\\

Let us now recall some basic facts from optimal transport theory. Let $\mathcal{P}(X)$ be the set of probability measures on $X$, $\mathcal{P}_{2}(X)$ the set of probability measures $\mu$ on $X$ with finite second moment, meaning that $\int_X \dist^2(x_o,x) \di \mu(x) < + \infty$ for some $x_o \in X$. We also write $\mathcal{P}_2^{a}(X,\meas)$ for the subset of $\mathcal{P}_2(X)$ made of those measures that are absolutely continuous with respect to $\meas$. The Wasserstein distance between two measures $\mu_{0},\mu_{1} \in \mathcal{P}_{2}(X)$ is by definition
\begin{equation}\label{eq:W}
W_{2}(\mu_{0},\mu_{1}) := \inf_\pi \left( \int_{X \times X} \dist(x_{0},x_{1})^2 \di \pi(x_{0},x_{1})  \right)^{1/2}
\end{equation}
where the infimum is taken among all the probability measures $\pi$ on $X\times X$ with first marginal equal to $\mu_{0}$ and second marginal equal to $\mu_{1}$. Any measure $\pi$ achieving the infimum in \eqref{eq:W} is called optimal coupling between $\mu_0$ and $\mu_1$. A standard result states that if the space $(X,\dist)$ is geodesic (meaning here that two points in $X$ can be joined by a globally distance minimizing curve), then the metric space $(\mathcal{P}_{2}(X),W_{2})$ is geodesic too. Finally, for any $N \in (1,+\infty)$, the $N$-Rényi entropy relative to $\meas$, denoted by $S_{N}(\cdot | \meas)$, is defined by:
\[
S_{N}(\mu | \meas) := - \int_{X} \rho^{1-\frac{1}{N}} \di \meas \qquad \forall \mu \in \mathcal{P}(X),
\]
 where $\mu = \rho \meas + \mu^{s}$ is the Lebesgue decomposition of $\mu$ with respect to $\meas$.\\

We are now in a position to introduce the $\RCD(K,N)$ condition.

\begin{definition}
Let $(X,\dist,\meas)$ be a Polish metric measure space and $K \in \setR$, $N \ge 1$ two parameters.\\

1. \cite{LottVillani,Sturm2006I,Sturm2006II} The space $(X,\dist,\meas)$ is called $\CD(K,N)$ if for any $\mu_0, \mu_1 \in \mathcal{P}^a_2(X,\meas)$ with respective densities $\rho_0,\rho_1$, there exists at least one $W_2$-geodesic $(\mu_t)_{t \in [0,1]}$ and an optimal coupling $\pi$ between $\mu_0$ and $\mu_1$ such that for any $N' \ge N$,
\[
S_{N'}(\mu_t | \meas) \le - \int_{X\times X} [\tau_{K,N'}^{(1-t)}(\dist(x_0,x_1))\rho_0^{-1/N'}(x_0) + \tau_{K,N'}^{(t)}(\dist(x_0,x_1))\rho_1^{-1/N'}(x_1)] \di \pi(x_0,x_1),
\]
where for any $\theta \ge 0$,
$$
\tau_t^{(K,N)} (\theta) :=
\begin{cases}
t^{\frac{1}{N}} \left( \frac{\sinh(t \theta \sqrt{-K/(N-1)})}{\sinh(\theta \sqrt{-K/(N-1)})} \right)^{1-\frac{1}{N}} & \text{if $K<0$,} \\
t & \text{if $K=0$,} \\
t^{\frac{1}{N}} \left( \frac{\sin(t \theta \sqrt{K/(N-1)})}{\sin(\theta \sqrt{K/(N-1)})} \right)^{1-\frac{1}{N}} & \text{if $K>0$ and $0< \theta < \pi \sqrt{(N-1)/K}$,} \\
\infty & \text{if $K>0$ and $ \theta \ge \pi \sqrt{(N-1)/K}$,}
\end{cases}
$$
if $N > 1$, and $\tau_t^{(K,N)}(\theta) = t$ if $N=1$.\\

2. \cite{AmbrosioGigliSavare-Duke} The space $(X,\dist,\meas)$ is called $\RCD(K,N)$ if it is both $\CD(K,N)$ and infinitesimally Hilbertian.
\end{definition}

It is worth mentioning that $\CD(K,N)$ (and thus $\RCD(K,N)$) spaces satisfy the local doubling and Poincaré properties:

\begin{enumerate}
\item[(i)] (local doubling condition) for any $R>0$ there exists $C_D=C_D(K,N,R)>0$ such that for any ball $B$ with radius $r \in (0,R)$,
$$
\meas(2B) \le C_D \meas(B);
$$
\item[(ii)] (local weak $(1,1)$-Poincaré inequality) for any $R>0$ there exists $C_P=C_P(K,N,R)>0$ such that for any $f \in \Lip(X)$ and any ball $B$ with radius $r\in(0,R)$,
$$
\fint_{B} |f - f_{B}| \di \meas \le C_P r  \fint_{2B} |\nabla f| \di \meas.
$$
\end{enumerate}
Here $2B$ is the ball with same center as $B$ but with doubled radius.

Moreover, any $\RCD(K,N)$ space has a notion of essential dimension in the sense that there exists a unique integer $\dimnew(X):=n\in [1,N]$ such that
\begin{equation}\label{eq:constant}
\meas(X\setminus \mathcal{R}_n\bigr)=0
\end{equation}
where $\displaystyle \mathcal{R}_n:=\{ x \in X \, : \mathrm{Tan}(X, \dist,\meas,x) =\{ \bigl (\mathbb{R}^k, \dist_{\mathbb{R}^k},\mathcal{L}^k/\omega_k,0\bigr) \}$ is the set of so-called $n$-regular points of $(X,\dist,\meas)$ and for any $x \in X$, $\mathrm{Tan}(X, \dist,\meas,x)$ is the set of tangents to $(X,\dist,\meas)$ at $x$ that is to say the collection of all
pointed metric measure spaces $(Y, \dist_Y,\meas_Y,y)$ such that, as $i\to\infty$, one has
$$
\left(X, \frac{1}{r_i}\dist,\frac{\meas}{\meas (B_{r_i}(x))},x\right) \stackrel{mGH}{\to} (Y, \dist_Y,\meas_Y,y) 
$$
for some infinitesimal sequence $(r_i)\subset (0,\infty)$, where $mGH$ stands for the measured Gromov-Hausdorff convergence. A first step towards this deep structural result was achieved by A.~Mondino and A.~Naber who established in \cite{MondinoNaber}  the a.e.~decomposition
$\meas(X \backslash \bigcup_{k=1}^{\lfloor N \rfloor} \mathcal{R}_k)=0$
of any $\RCD(K,N)$ space $(X,\dist,\meas)$. Several groups of authors \cite{GigliPasqualetto,DePhilippisMarcheseRindler,KellMondino} subsequently refined this result by proving the rectifiability of $\RCD(K,N)$ spaces as metric measure spaces. Building upon this and a careful analysis of the regularity of Regular Lagrangian Flows of Sobolev vector fields, E.~Brué and D.~Semola finally proved \eqref{eq:constant} by a suitable contradiction argument \cite{BrueSemola}.
\\

Finally, let us provide some examples of $\RCD(K,N)$ spaces.
\begin{itemize}
\item Smooth Riemannian manifolds $(M,g)$ equipped with their canonical Riemannian distance $\dist_g$ and volume measure $\vol_g$ with dimension $n$ bounded from above by $N$ and Ricci curvature bounded from below by $K(n-1)$ are $\RCD(K,N)$ spaces. For instance, the unit sphere in $\setR^n$ equipped with the Riemannian metric induced by the ambient Euclidean metric is $\RCD(1,n-1)$. 

\item Weighted smooth $n$-dimensional Riemannian manifolds $(M,\dist_g,e^{-V}\vol_g)$ with $V \in C^2(M)$ satisfying $
\Ric + \, \mathrm{Hess}_V - \frac{1}{N-n}\nabla V \otimes \nabla V \ge K g$ are $\RCD(K,N)$ spaces. For instance,
$([0,\pi],\dist_{eucl},\sin^{N-1}(r)\di r)$ is a $\RCD(1,N)$ space.

\item Two-dimensional cones with angle less than or equal to $2\pi$ are $\RCD(0,2)$ spaces. This is probably the simplest example of non-smooth $\RCD$ space. Elaborated variations on this simple example include cones over $\RCD$ spaces \cite{Ketterer} and stratified spaces \cite{BertrandKettererMondelloRichard}.

\item The graph of any Lipschitz function $f:\setR^n \to \setR$ equipped with the length distance and the $n$-dimensional Hausdorff measure is $\RCD(K,n)$, where $K$ depends on the Lipschitz constant of the function.

\item Two important classes of possibly highly non-smooth metric measure spaces are also $\RCD(K,N)$ spaces. These are the class of Ricci limit spaces and Alexandrov spaces. Spaces in this latter class are metric spaces with a synthetic notion of sectional curvature bounded from below and are naturally endowed with a Hausdorff measure of integer dimension, see \cite{BuragoGromovPerelman,Petrunin}.
\end{itemize}

\hfill

\textbf{Heat kernel of $\RCD(K,N)$ spaces}

Let us recall that a classical gradient flow (in a Hilbert space $H$) is the solution of an ordinary differential equation of type
\begin{equation}\label{eq:gradientflow}
\begin{cases}
& x' = -\nabla F(x)\\
& x(0) = \overline{x}
\end{cases}
\end{equation}
where $F: H \to \mathbb{R}\cup\{+\infty\}$ is a lower semicontinuous function satisfying some regularity assumption (say $C^{1,1}$). In case the only regularity assumption on $F$ is convexity, one can still give a meaning to \eqref{eq:gradientflow} by introducing the subdifferential of $F$, defined as
\[ \partial F(x) := \{ p \in H : \forall y \in H, \, F(y) \ge F(x) + \scal{p}{y-x}_H \}\]
for any $x \in H$. Then we call gradient flow of $F$ starting at $\overline{x} \in H$ any locally absolutely continuous curve $x : (0,+\infty) \to H$ such that 
\begin{equation}\label{eq:gradientflow2}
\begin{cases}
& x'(t) \in - \partial F(x(t)) \qquad \text{for a.e.} \, t \in (0,+\infty)\\
& \|x(t) - \overline{x}\|_{H} \to 0\qquad \, \, \, \text{when $t \to 0$.}
\end{cases}
\end{equation}
The Komura-Brézis theorem \cite{Komura,Brezis,AmbrosioBrueSemola} states that for any $\overline{x}$ in the closure of the finiteness domain of $F$, there exists a unique gradient flow of $F$ starting at $\overline{x}$.

Considering a metric measure space $(X,\dist,\meas)$, it is easily checked that $\Ch$ is convex and lower semicontinuous with respect to the $L^2(X,\meas)$-norm. Therefore, assuming that $(X,\dist,\meas)$ is infinitesimally Hilbert, the Komura-Brézis theorem applies and provides a family of maps $$ P_t : \{\partial \Ch (\cdot) \neq 0 \} \subset L^2(X,\meas) \to L^2(X,\meas), \qquad t >0, $$ defined by $P_t(f):=f(t)$ for any $f \in L^2(X,\meas)$ where $f(\cdot)$ is the gradient flow of $\Ch$ starting from $f$. This family is called heat flow of $(X,\dist,\meas)$ because for any $f$ such that $\partial \Ch (f) \neq 0$, if we set $-\Delta f$ as the element with minimal norm in $\partial \Ch (f)$, it can be shown that
$$
\frac{\di}{\di t} P_tf = - \Delta P_tf
$$
holds for a.e.~$t>0$. Moreover, using the infinitesimal Hilbertianity of the space, one can show that:
\begin{enumerate}
\item the maps $(P_t)_{t>0}$ are linear,

\item the Laplacian $\Delta$ coincides with the linear operator defined through integration by parts:
\begin{align*}
\mathcal{D}(\Delta) := & \{  f \in H^{1,2}(X,\dist,\meas) \, : \, \, \text{there exists} \, \, h=:\Delta f \in L^2(X,\meas) \, \, \text{such that} \\
&  \int_X \langle \nabla f, \nabla g \rangle \di \meas = - \int_X h g \di \meas  \, \, \, \text{for all} \, \, \,  g\in H^{1,2}(X,\dist,\meas) \, \},
\end{align*}

\item the limit $$
\langle \nabla f, \nabla g \rangle:=\lim\limits_{\eps \downarrow 0} \frac{|\nabla (f + \eps g)|_*^2 - |\nabla f|_*^2}{2\eps}$$ defines a symmetric bilinear form on $H^{1,2}(X,\dist,\meas)\times H^{1,2}(X,\dist,\meas)$ taking values in $L^1(X,\meas)$ and
$$
\mathcal{E} (f,g) := \int_X \langle \nabla f, \nabla g \rangle \di \meas
$$
defines a strongly local Dirichlet form with domain $H^{1,2}(X,\dist,\meas)$ such that:
$$\mathcal{E}(f,f)=\Ch(f) \qquad \forall f \in H^{1,2}(X,\dist,\meas).$$
\end{enumerate}

When $(X,\dist,\meas)$ is an $\RCD(K,N)$ space, then $\cE$ is also a regular Dirichlet form whose associated intrinsic distance coincides with $\dist$, see \cite{AmbrosioGigliSavare-Duke}.
Then the works of K.-T.~Sturm on such Dirichlet forms \cite{SturmI, SturmII, SturmIII} provide the existence of a heat kernel for $\Ch$, meaning in this context a locally Lipschitz function $p:X\times X \times (0,+\infty) \to (0,+\infty)$ symmetric with respect to its first two variables such that for any $t>0$ and $f \in L^2(X,\meas)$,
$$P_t f (x) = \int_X p(x,y,t) f(y) \di \meas(y)  \qquad \text{for $\meas$-a.e.~$x \in X$}.$$
Note that Sturm's results hold in the class of metric measure spaces satisfying the doubling and Poincaré properties, so the existence of a heat kernel on $(X,\dist,\meas)$ uses both the $\CD(K,N)$ condition (notably to ensure the validity of these two conditions) and the infinitesimally Hilbertian condition.\\

\textbf{Riemannian metrics on $\RCD(K,N)$ spaces}

A notion of Riemannian metric can be formulated on $\RCD(K,N)$ spaces thanks to the abstract calculus developed by N.~Gigli in \cite{Gigli}. There is shown that on an $\RCD(K,N)$ space $(X,\dist,\meas)$ can be defined:
\begin{itemize}
\item a space of square integrable vector fields $L^2(TX)$ equipped with a natural norm $\|\cdot\|_{L^2(TX)}$ such that any $f \in H^{1,2}(X,\dist,\meas)$ defines an element $\nabla f \in L^2(TX)$ with $\|\nabla f\|_{L^2(TX)} = |\nabla f|_*$,
\item its dual $L^2(T^*X)$,
\item their tensor products $L^2(TX) \otimes L^2(TX)$, $L^2(T^*X) \otimes L^2(T^*X)$, $L^2(TX) \otimes L^2(T^*X)$,
\item a local Hilbert-Schmidt norm $\| \cdot \|_{HS} : L^2(T^*X) \otimes L^2(T^*X) \to L^0(X,\meas)$, where $L^0(X,\meas)$ is the set of $\meas$-measurable functions on $X$.
\end{itemize}
Then a Riemannian metric on $(X,\dist,\meas)$ is by definition a symmetric bilinear form $\bar{g} : L^2(TX) \times L^2(TX) \to L^0(X,\meas)$ that is  $L^\infty(X,\meas)$-linear, meaning that$$\bar{g}(\chi V,W)=\chi\bar{g}(V,W)$$ for any $\chi \in L^\infty(X,\meas)$ and $V,W \in L^2(TX)$, and non-degenerate, that is to say:$$\bar{g}(V,V)>0\quad \text{$\meas$-a.e.~on $\{|V|>0\}$} \quad \text{for all $V \in L^2(TX)$}.$$
Any Riemannian metric $\bar{g}$ can be represented by a unique element $\bar{\mathfrak{g}}$ in $L^2(T^*X) \otimes L^2(T^*X)$ 
singled out by the following property:
$$
\langle \bar{\mathfrak{g}}, \sum_i \chi_i \nabla f_i^1 \otimes \nabla f_i^2 \rangle_{dual} = \sum_i \chi_i \bar{g}(\nabla f_i^1,\nabla f_i^2)
$$
for any finite collection  $\chi_i \in L^\infty(X,\meas)$, $ f_i^1, f_i^2 \in H^{1,2}(X,\dist,\meas)$ where $\langle \cdot, \cdot \rangle_{dual}$ is the duality pairing. Moreover, there exists a unique Riemannian metric $g$ such that
$$
g(\nabla f_1,\nabla f_2) = \langle \nabla f_1, \nabla f_2 \rangle \qquad \text{$\meas$-a.e.}
$$
for any $f_1, f_2 \in H^{1,2}(X,\dist,\meas)$. This metric is called \textit{canonical Riemannian metric} of $(X,\dist,\meas)$.\\

\textbf{Spectral and heat kernel embeddings}

When $(X,\dist,\meas)$ is a compact $\RCD(K,N)$ space, one can show that the linear operator $-\Delta$ has a discrete spectrum $0=\lambda_0<\lambda_1 \le \lambda_2 \le \ldots \to + \infty$, the eigenfunctions of $-\Delta$ all admit a Lipschitz representative, and the heat kernel supports a spectral decomposition just like in the case of smooth Riemannian manifolds:
\begin{equation}\label{eq:spectralRCD}
p(\cdot,\cdot,t) = \sum_{i \ge 0} e^{- \lambda_i t} \varphi_i(\cdot) \varphi_i (\cdot) \qquad \text{in $C(X \times X)$},
\end{equation}
$$
p(\cdot,y,t) = \sum_{i \ge 0} e^{- \lambda_i t} \varphi_i(y) \varphi_i(\cdot) \qquad \text{in $H^{1,2}(X,\dist,\meas)$} \qquad \text{for any $y \in X$.}
$$
Thus the Bérard-Besson-Gallot embeddings can be defined in a direct way on $(X,\dist,\meas)$. Let $\mathcal{B}(X,\dist,\meas)$ be the set of orthonormal basis of $L^2(X,\meas)$ made of normalized eigenfunctions of $-\Delta$ listed in increasing order of corresponding eigenvalues. Then for any $t>0$ and $a=(\phi_i)_{i} \in \mathcal{B}(X,\dist,\meas)$, we can set
$$
\Psi_t^{a}\,  : \, \, \,  
\begin{cases}
\,\, X & \to \quad l^2\\
\,\, \, x & \mapsto \quad c_n t^{(n+2)/4} (e^{-\lambda_i t /2} \phi_i(x))_{i \ge 1}
\end{cases}
$$
with $c_n=\sqrt{2}(4 \pi)^{n/4}$. However, the heat kernel embeddings 
$$
\Phi_t\,  : \, \, \,  
\begin{cases}
\,\, X & \to \quad L^2(X,\meas)\\
\,\, \, x & \mapsto \quad p(x,\cdot,t),
\end{cases}
$$
where $t>0$, are easier to handle in this context. Indeed, the study of $\RCD(K,N)$ spaces often relies on blow-up arguments where the local analysis on $(X,\dist,\meas)$ at $x \in X$ is observed through the behavior of the rescaled spaces $(X,\sqrt{t}^{-1}\dist,\meas(B_{\sqrt{t}}(x))^{-1}\meas, x)$ when $t \downarrow 0$. Working with the heat kernel in this case is especially convenient because of the simple scaling formula $$p_t(x,y,1)=\meas(B_{\sqrt{t}}(x)) p(x,y,t)$$ where $p_t$ is the heat kernel of $(X,\sqrt{t}^{-1}\dist,\meas(B_{\sqrt{t}}(x))^{-1}\meas)$. Note that for any $a=\{\phi_i\}_i \in \mathcal{B}(X,\dist,\meas)$, the spectral decomposition \eqref{eq:spectralRCD} implies
$$
c_n t^{(n+2)/4}\Lambda^a(\Phi_{t/2}(x)) = \Psi_t^{a}(x)
$$
for any $x \in X$ and $t>0$, where $\Lambda^a$ is the isomorphism $L^2(X,\meas) \ni f=\sum_i f_i \phi_i \mapsto \{f_i\}_i \in l^2$, so the properties of $\Phi_t$ can be deduced from those of $\Psi_t^{a}$ and vice-versa.

It follows from the same proof as in the Riemannian case that the maps $\Phi_t$ (and then $\Psi_t^{a}$) are Lipschitz embeddings for any $t>0$. Moreover, when $(X,\dist,\meas)$ is a smooth Riemannian manifold $(M,\dist_g,\vol_g)$, it is easily checked that
\begin{equation}\label{eq:riemform}
\text{$D\Phi_t(x)\cdot v$ coincides with the square integrable function $y \mapsto g_x( \nabla p(\cdot,y,t)(x),v)$}
\end{equation}
for any $x \in X$ and $v \in T_xM$. Writing $v$ as the initial velocity of a smooth curve $\gamma:[0,1]\to M$ emanating from $x$, this writes as
$$
[D\Phi_t(x)\cdot v](\cdot) = \left. \frac{\di}{\di s} \right|_{s = 0} p(\gamma(s),\cdot,t).
$$
This Riemannian formula extends to a first-order differentiation formula in the general $\RCD(K,N)$ setting, see \cite[Prop.~5.2.1]{TewodrosePhD}. Moreover, following the Riemannian observation \eqref{eq:riemform}, one can define pull-back Riemannian metrics on $(X,\dist,\meas)$ induced by the heat kernel embeddings $(\Phi_t)_{t>0}$.

\begin{proposition}[Pull-back metrics]
For any $t>0$, setting
\begin{equation}
g_t (V_1, V_2)(\cdot) := 
\int_X\langle \nabla p(\cdot, y, t), V_1\rangle \langle \nabla p(\cdot, y, t), V_2\rangle\di\meas(y)
\end{equation}
for any $V_1, V_2 \in L^2(TX)$ defines a Riemannian metric on $(X,\dist,\meas)$.
\end{proposition}

Next is one of the main theorems in \cite{AmbrosioHondaPortegiesTewodrose} that brings information on the asymptotic behavior of the heat kernel embeddings when $t \downarrow 0$. Note that the rescaling factor $t^{(n+2)/2}$ in the classical Bérard-Besson-Gallot spectral embeddings leads to two possible rescalings in the $\RCD$ context: $t^{(n+2)/2}$ or $t \meas(B_{\sqrt{t}}(x))$. The second one turns out more natural since it takes into account possible degeneracy
points where the measure might not have an Euclidean like infinitesimal behavior, but we provide a convergence result for both.

\begin{theorem}\label{th:RCDconvergence}
Let $(X,\dist,\meas)$ be a compact $\RCD(K,N)$ space with essential dimension $n$. Set $\hat{g}_t:=t \meas(B_{\sqrt{t}}(\cdot)) g_t$ and $\tilde{g}_t:=t^{(n+2)/2}g_t$ for any $t>0$. Then there exists a dimensional constant $c_n>0$ such that when $t \downarrow 0$,
\begin{enumerate}
\item the weak convergence $\hat{g}_t \to c_ng$ holds in the sense that $\hat{g}_t(V,V) \to c_n g(V,V)$ for any $V \in L^2(TX)$ in the weak topology of $L^1(X,\meas)$,

\item the strong convergence $\hat{g}_t \to c_ng$ holds in the sense that $\lim\limits_{t \to 0} \||\hat{\mathfrak{g}}_t - \mathfrak{g}|_{HS}\|_{L^2} = 0$,

\item the weak convergence $\tilde{g}_t \to c_nF(\cdot)g$ holds, where $F(\cdot)$ is the inverse of the density of $\meas$ with respect to $\haus^n$ (see \cite[Th.~4.1]{AmbrosioHondaTewodrose}),

\item the strong convergence $\tilde{g}_t \to c_nF(\cdot)g$ holds.
\end{enumerate}
\end{theorem}

On a Riemannian manifold $(M,g)$, the Riemannian distance between two points $x$ and $y$ is set as
\begin{equation}\label{eq:dist}
\dist_g(x,y):=\inf \left\{\int_0^1 \sqrt{g_{\gamma(t)}(\gamma'(t),\gamma'(t))} \di t : \gamma \in \mathrm{Adm}(x,y)\right\}
\end{equation}
where $\mathrm{Adm}(x,y)$ is the set of $C^1$ maps $\gamma : [0,1] \to M$ such that $\gamma(0)=x$ and $\gamma(1)=y$. When $M$ is compact, a simple proof based on the Arzelà-Ascoli theorem shows that if $\{g_t\}_{t>0}$ is a family of smooth Riemannian metrics converging uniformly to $g$ as $t \downarrow 0$, then $\dist_{g_t} \to \dist_g$ pointwise. On a compact $\RCD(K,N)$ space, the picture is different because, as far as the author knows, no notion of a vector field along a curve exists yet in this context, hence \eqref{eq:dist} must be given an appropriate meaning which is still missing. For this reason, to turn a Riemannian metric on a compact $\RCD(K,N)$ space into a distance remains a problem whose solving may help tackling the following one:\\

\textbf{Open problem 3.} Can one turn the convergence results for Riemannian metrics of Theorem \ref{th:RCDconvergence} into convergence results for suitably associated distances?\\

It is very natural to ask how sensitive the heat kernel embeddings might be to measured Gromov-Hausdorff perturbations of the space $(X,\dist,\meas)$. The next theorem (\cite[Th.~5.19]{AmbrosioHondaPortegiesTewodrose}) provides an answer to this question.

\begin{theorem}
Let $\mathrm{CRCD}(K,N)$ be the set of compact $\RCD(K,N)$ spaces equipped with the mesured Gromov-Hausdorff topology and $\mathrm{CMS}$ the set of compact metric spaces equipped with the Gromov-Hausdorff distance. Then
$$
F\,  : \, \, \,  
\begin{cases}
\,\, \mathrm{CRCD}(K,N) \times (0,+\infty) & \to \quad \mathrm{CMS}\\
\qquad \quad ((X,\dist,\meas),t) & \mapsto \quad (\Phi_t(X),\dist_{L^2(X,\meas)})
\end{cases}
$$
is a jointly continuous map.
\end{theorem}

\begin{remark}
It must be underlined that $(\Phi_t(X),\dist_{L^2(X,\meas)}) \stackrel{\dist_{GH}}{\to} (X,\dist)$ when $t\downarrow 0$ is still unknown. This is related to the issues raised before Open Problem 3.
\end{remark}

Finally, let us point out that a truncated and quantitative version of Theorem \ref{th:RCDconvergence} holds in the context of \textit{non-collapsed} $\RCD(K,N)$ spaces which are, by definition, $\RCD(K,N)$ spaces $(X,\dist,\meas)$ with $\meas=\haus^{\lfloor N \rfloor}$: see \cite[Th.~6.9]{AmbrosioHondaPortegiesTewodrose}.\\

Let us now sketch the proof of $\textit{(1)}$ in Theorem \ref{th:RCDconvergence}. Take $V \in L^2(TX)$. We must show
$$
\hat{g}_t(V,V) \xrightharpoonup{L^1} c_n g(V,V).
$$
By the Vitali-Hahn-Saks and Dunford-Pettis theorems, this amounts to showing
$$
\int_A\hat{g}_t(V,V) \di \meas \to c_n \int_A g(V,V) \di \meas
$$
for any Borel set $A \subset X$.
But this is a consequence of proving
\begin{equation}\label{eq:toprove}
\int_{A_1} \int_{A_2} t \meas(B_{\sqrt{t}}(x)) \langle \nabla p(\cdot,y,t)(x),V(x)\rangle^2 \di \meas(x) \di \meas(y) \to c_n \int_{A_1 \cap A_2} g(V,V) \di \meas
\end{equation}
for any Borel sets $A_1, A_2 \subset X$, as revealed by taking $A_1=X$ and $A_2=A$ and using Fubini's theorem.
For $L>0$ and $t>0$ fixed, we split $A_2$ into two parts:
\begin{align*}
& \int_{A_1} \int_{A_2} t \meas(B_{\sqrt{t}}(x)) \langle \nabla p(\cdot,y,t)(x),V(x)\rangle^2 \di \meas(x) \di \meas(y)\\
& =\quad  \int_{A_1} \int_{A_2\cap B_{L\sqrt{t}}(y)} \ldots  \quad + \quad \int_{A_1} \int_{A_2\backslash B_{L\sqrt{t}}(y)} \ldots \quad \cdot
\end{align*}
Using the sharp Gaussian estimates on the heat kernel established by R.~Jiang, H.~Li and H.-C.~Zhang in \cite{JiangLiZhang}, we get
\begin{equation}\label{eq:21}
\left|\int_{A_1} \int_{A_2\backslash B_{L\sqrt{t}}(y)} \ldots \quad \right| \le C(L)
\end{equation}
where $C(L) \to 0$ when $L\to +\infty$. Thus all the relevant information is contained in the other part of the integral. To deal with this latter, set $\dist_t := \sqrt{t}^{-1} \dist$ and $\meas_t:=\meas(B_{\sqrt{t}}(z))^{-1} \meas$. Let $\omega_n$ be the volume of the unit ball in $\setR^n$ equipped with the Lebesgue measure, and $\hat{\haus}^n:=\haus^n/\omega_n$. Then the idea is to replace $V$ by $\nabla f$ for a suitable function $f\in H^{1,2}$ chosen so that for all $y \in \mathcal{R}_n$, the rescalings
$$
f_{\sqrt{t},y} = \frac{1}{\sqrt{t}}( f-(f)_{\sqrt{t},y}) \in H^{1,2}(X,\dist_t,\meas_t),
$$
where $(f)_{\sqrt{t},y}$ denotes the $\meas$-mean-value of $f$ over the ball $B_{\sqrt{t}}(y)$, converge in a suitable sense to a Lipschitz and harmonic function $\hat{f}$ on $(\setR^n,\dist_{eucl},\hat{\haus}^n)$ such that
$$
\sum_{i=1}^n \left|\frac{\partial \hat{f}}{\partial x_j} \right|^2 = g(V,V)^2(y).
$$
Then if $\hat{p}_e$ is the heat kernel of $(\setR^n,\dist_{eucl},\hat{\haus}^n)$, we get
\begin{align*}
& \int_{B_{L \sqrt{t}}(y)} t \meas(B_{\sqrt{t}}(x)) \langle \nabla p(\cdot,y,t)(x), \nabla f(x) \rangle^2 \di \meas(x)
 \\
= \quad & \int_{B_L^{\dist_t}(y)} \meas_t(B_1^{\dist_t}(x)) \langle \nabla p_{\sqrt{t}}(\cdot,y,1)(x), \nabla f_{\sqrt{t},y}(x) \rangle^2 \di \meas_t(x)\\
\xrightarrow{t \downarrow 0} \, \, \, & \int_{B_L(0_n)} \hat{\haus}^n(B_1(x)) \langle \nabla \hat{p}_e(\cdot,0_n,1)(x), \nabla \hat{f}(x) \rangle^2 \di \hat{\haus}^n(x) \\
= \quad &  \, \, c_n(L) \sum_{j=1}^n \left| \frac{\partial \hat{f}}{\partial x_j} \right|^2 = c_n(L) \,  g(V,V)^2(y)
\end{align*}
with $c_n(L) \to c_n$ when $L \to +\infty$. We obtain in a similar manner
$$
\int_{A_2 \cap B_{L\sqrt{t}}(y)} t \meas(B_{\sqrt{t}}(x)) \langle \nabla p(\cdot,y,t)(x),V(x)\rangle^2 \di \meas(x) \xrightarrow{t \downarrow 0}  c_n(L) g(V,V)(y) 1_{A_2}(y)
$$
for $\meas$-a.e.~$y \in X$, unifomly in $y$.  Therefore, the convergence is preserved when we integrate with respect to $y \in A_1$, thus
$$
\int_{A_1} \int_{A_2\cap B_{L\sqrt{t}}(y)} t \meas(B_{\sqrt{t}}(x)) \langle \nabla p(\cdot,y,t)(x),V(x)\rangle^2 \di \meas(x) \di \meas(y)
$$
$$
\xrightarrow{t \downarrow 0} c_n(L) \int_{A_1\cap A_2} g(V,V) \di \meas. 
$$
This combined with \eqref{eq:21} implies \eqref{eq:toprove} by letting $L$ tend to $+\infty$.

\section{Applications to data analysis}

In this section, we present two manifold learning algorithms that are based on the spectral embedding theorems described in the previous section. A manifold learning algorithm takes in input a data set represented as a point cloud in $\setR^D$ and gives in output a lower dimensional representation of the data set, provided the original point cloud $\{x_1,\ldots,x_N\} \subset \setR^D$ lies on (or near) a smooth submanifold $M$ of $\setR^D$. Note that in this setting, only the original point cloud is known: the submanifold $M$ as well as its dimension are unknown a priori.

A common feature of these algorithms is the construction of a weighted graph $(\cV,\cE,w)$ from the point cloud $\{x_1,\ldots,x_N\}$ and the study of the eigenvalues and eigenvectors of suitable operators on this graph.

Let us fix some notation for this section. $M$ is a $d$-dimensional submanifold of $\setR^D$ that we may sometimes explicitely assume to be closed. $M$ is equipped with the Riemannian metric inherited from the ambiant Euclidean metric, and we denote by $\dist$, $\vol$ and $\Delta$ the associated canonical distance, volume measure and Laplace-Beltrami operator, respectively.

\subsection{Laplacian Eigenmaps}

The first algorithm we present is due to M. Belkin and P. Niyogi \cite{BelkinNiyogi}.\\

\textbf{The setting}

Let $x_1,\ldots,x_N \subset \setR^D$ be lying on a smooth $d$-dimensional submanifold $M$ of $\setR^D$. Let $(\cV,\cE)$ be the graph constructed from $X:=\{x_1,\ldots,x_N\}$ by setting $\cV:=X$ and building $\cE$ by means of one of the two following options:
\begin{enumerate}
\item choose $\eps>0$ and define $\cE$ as the collection of couples $(x_i,x_j) \in \cV \times \cV$ such that $\|x_i - x_j \|_{\setR^D}\le \sqrt{\eps}$,

\item choose an integer $n$ between $1$ and $N$ and define $\cE$ as the collection of couples $(x_i,x_j) \in \cV \times \cV$ such that for any $i$, the point $x_j$ is among the $n$ nearest neighbors of $x_i$, i.e.~the value $\|x_i - x_j \|_{\setR^D}$ is among the $n$ smallest values of the set $\{\|x_i - x \|_{\setR^D} \, : \, x \in X\}$.
\end{enumerate} 

Write $i\sim j$ as a shorthand for $(x_i,x_j) \in \cE$ and set $\deg(i)$ as the degree of the point $x_i$ that is the number of points $x_j$ such that $i \sim j$. Assume $(\cV,\cE)$ to be connected: if this is not the case, the algorithm can be performed on each connected component.\\

\textbf{A rough explanation}

A classical definition of the Laplacian on a graph $(\cV,\cE)$ is obtained by mimicking the property of the classical Euclidean Laplacian to measure the difference between a function and its mean-value on small balls: by Taylor's expansion, any function $f \in C^2(\setR^n)$ is such that
$$
f(x) - \fint_{B_r(x)} f = c_n r^2 \Delta f(x) + o(r^2), \qquad r \downarrow 0,
$$
for any $x \in \setR^n$, with $c_n=(2n+4)^{-1}$. Then the Laplacian on $(\cV,\cE)$ is usually defined by setting
$$
\Delta_{(\cV,\cE)}f(x_i):=f(x_i)-\frac{1}{\deg(i)} \sum_{j\sim i} f(x_j)
$$
for any $f:\cV \to \setR$ and $x_i \in \cV$. The operator $\Delta_{(\cV,\cE)}$ is sometimes called \textit{normalized} Laplacian of $(\cV,\cE)$, in opposition to the \textit{unnormalized} one defined by
$$
\Delta_{(\cV,\cE)}'f(x_i):=\deg(i) f(x_i)- \sum_{j\sim i} f(x_j).
$$
The addition of a weight $w$ to $(\cV,\cE)$ modifies the geometry of the graph in the sense that it measures the proximity between points: if $w(i,j)$ is big, then $x_i$ and $x_j$ must be understood as close, while if $w(i,k)$ is small, $x_i$ and $x_k$ must be understood as far apart. Then the contribution of $f(x_j)$ to the analogue of $\Delta_{(\cV,\cE)}f(x_i)$ in this weighted context should be more important than the contribution of $f(x_k)$. This is reflected in the following definitions:
$$
\Delta_{(\cV,\cE,w)}f(x_i)=f(x_i)-\frac{1}{\dist(i)} \sum_{j\sim i} w(i,j)f(x_j)
$$
and 
$$
\Delta_{(\cV,\cE,w)}' f(x_i)=\dist(i) f(x_i)- \sum_{j\sim i} w(i,j)f(x_j)
$$
for any $f:\cV \to \setR$ and $x_i \in \cV$, where $\dist(i):=\sum_{j \sim i} w(i,j)$.

Let us anticipate on the next paragraph and point out that from this perspective, the operators $L_t$ considered by M.~Belkin and P.~Niyogi are the normalized weighted Laplacians of the weighted graphs $(\cV,\cE,w_t)$ where
\begin{equation}\label{eq:w}
w_t(i,j):=
\begin{cases}
\,\,e^{-\frac{\|x_i - x_j\|^2_{\setR^D}}{t}} & \text{if $i \sim j$},\\
\qquad \,\, \, 0 & \text{otherwise},
\end{cases}
\end{equation}
for any $x_i,x_j \in \cV$. Let us explain the choice of this weight. Recall that $M$ is a $d$-dimensional sublmanifold of $\setR^D$. As well-known, the Laplace-Beltrami operator on $M$ can be expressed in terms of the heat kernel: for any $f \in C^2(M)$ and $x \in M$,
\[
\Delta_M f (x) = \left. \frac{\partial}{\partial t}\right|_{t=0} \int_M p(x,y,t) f(y) \di \vol(y),
\]
so we can make the rough approximation
\[
\Delta_M f (x) \approx \frac{1}{t} \left(f(x) - \int_M p(x,y,t) f(y) \di \vol(y) \right)
\]
for $t>0$ sufficiently small. When $t \to 0$, the heat kernel $p(x,\cdot,t)$ tends to the Dirac mass at $x$. In particular, it localizes so strongly that for some small $\eps>0$, one can consider as negligible the values of $p(x,\cdot,t)$ outside the ball $B_\eps(x)$. By the Minakshisundaram-Pleijel expansion and the simple observation
\begin{equation}\label{eq:restriction}
\dist(x,y) = \|x-y\|_{\setR^D} + o( \|x-y\|_{\setR^D}), \qquad y \to x,
\end{equation}
which is a consequence of $M$ being equipped with the restriction of the Euclidean metric, we get that $p(x,\cdot,t)$ can be approximated by the Gaussian term
\[
\frac{1}{(4\pi t)^{d/2}} e^{-\frac{\|x-\cdot\|_{\setR^D}^2}{4t}}1_{B_\eps(x)}(\cdot).
\]
Thus we may write
$$
\Delta_M f(x) \approx \frac{1}{t} \left(f(x) - \frac{1}{(4\pi t)^{d/2}} \int_{B_\eps(x)} e^{-\frac{\|x-y\|_{\setR^D}}{4t}} f(y) \di \vol(y) \right).
$$
Now if $\{x_1,\ldots,x_N\}$ is a point cloud lying on $M$, defining the associated graph $(\cV,\cE)$ by choosing, for instance, the first option to construct $\cE$ with the above parameter $\eps$, at each point $x_i \in \cV$ we can approximate the integral in the previous expression by a Riemann sum over the neighbors of $x_i$:
$$
\Delta_M f(x_i) \approx \frac{1}{t} \left(f(x_i) - \frac{1}{(4\pi t)^{d/2}} \frac{1}{\deg(x_i)} \sum_{j \sim i} e^{-\frac{\|x_i-x_j\|_{\setR^D}}{4t}}f(x_j) \right).
$$
Since $d$ is unknown a priori, one may replace the dimensional coefficients $(4\pi t)^{d/2}\deg(x_i)$ by unknown varying coefficients $\alpha_i$. In order to have an operator vanishing on constant functions, we must choose $\alpha_i=\deg(x_i)^{-1}\sum_{j \sim i} e^{-\|x_i-x_j\|_{\setR^D}^2(4t)^{-1}}$ for any $i$. This finally leads to
$$
\Delta_M f(x_i) \approx \frac{1}{t} L_tf(x_i).
$$
This suggests a correspondance between the eigenvalues $\lambda_i$ and eigenfunctions $\phi_i$ of $\Delta_M$ and the eigenvalues $\lambda_i^t$ and eigenvectors $\phi_i^t$ of $L_t$ of the following type:
$$
\lambda_i \approx \frac{1}{t}\lambda_i^t \qquad \text{and} \qquad \phi \approx \phi_i^t.
$$
\hfill

\textbf{The algorithm}

Let $t>0$ be a parameter. Define $w_t$ accroding to \eqref{eq:w}. Let $W_t$ be the $N \times N$ matrix whose $(i,j)$-th entry is $w_t(i,j)$ and $D_t$ the diagonal square matrix of same size as $W_t$ with $i$-th diagonal term defined as $d_t(i):=\sum_{j\sim i} w_t(i,j)$; note that $d_t(i)$ is sometimes called \textit{weighted degree} of $x_i$. Set
\[
L_t:= I_N - D_t^{-1}W_t
\]
where $I_N$ is the $N\times N$ identity matrix, and note that for any $v=(v_1,\ldots,v_N)$, the $i$-th coordinate of the vector $L_tv$ is
\[
v_i - \frac{1}{d_t(i)}\sum_{j \sim i} w_t(i,j)v_j.
\]
Since $L_t$ is equivalent to a symmetric matrix: $$L_t=D_t^{-1/2}(I-D_t^{-1/2}W_tD_t^{-1/2})D_t^{-1/2},$$
then it admits eigenvalues that we list in increasing order:
\[
\lambda_0^t \le \lambda_1^t \le \ldots \le \lambda_{N-1}^t.
\]
Note that $\lambda_0^t=0$ corresponds to the eigenspace generated by the constant vector $(1,\ldots,1)$.

For any $i$ between $1$ and $N_1$, write $\phi_i^t$ for the normalized eigenvector of $L_t$ corresponding to the eigenvalue $\lambda_i^t$ -- here by normalized we mean that $\|\phi_i^t\|_{\setR^N}=1$. Then the Laplacian Eigenmaps of the data point cloud are the embeddings $\Phi_m^t : X \to \setR^m$, where $1 \le m \le N-1$, defined by
\[
\Phi_m^t(x) = (\phi_1^t(x),\ldots,\phi^t_m(x))
\]
for any $x \in X$. The Laplacian Eigenmaps may be seen as the discrete counterpart of the truncated Bérard-Besson-Gallot embeddings \eqref{eq:truncatedembedding}. As so they provide, for $t$ small enough, a faithful representation of the point cloud into a low dimensional Euclidean space. 

\subsection{Convergence of the Laplacian Eigenmaps}

In \cite{BelkinNiyogi2}, M.~Belkin and P.~Niyogi studied the behavior of their algorithm when $N$ goes to $+\infty$ by replacing the fixed data points $x_1,\ldots,x_N$ by random variables $X_1, \ldots, X_N$. They proved that in this case, the eigenvalues and eigenvectors appearing in their algorithm converge to those of minus the rescaled Laplace-Beltrami operator $-\underline{\Delta}:=-\Delta/\vol(M)$, provided $M$ is closed, what we assume from now on.

Let $(\Omega,\mathcal{A},\mathbb{P})$ be a fixed probability space. All the random variables considered in the sequel have domain $\Omega$. We recall that a sequence of real-valued random variables $\{Y_i\}_{i\ge 1}$ converges in probability to another real-valued random variable $Y$ if $\mathbb{P}(|Y_i-Y|\ge\eps)\to 0$ for any $\eps>0$, and that it converges almost surely to $Y$ if there exists a $\mathbb{P}$-negligible set $N \subset \Omega$ such that $Y_n(\omega) \to Y(\omega)$ for any $\omega \in \Omega \backslash N$.

Let $\{X_i\}_{i\ge 1}$ be independant and identically distributed random variables on $M$ with law $\underline{\vol}:=\vol/\vol(M)$. For any integer $N \ge 1$ and any $t>0$, let $(\cV,\cE,w_t)$ be the random graph and
$$
\mathrm{L}_{t,N}:= I_N  - \mathrm{D}_{t,N}^{-1} W_{t,N} 
$$
the random matrix obtained by applying the Laplacian Eigenmaps process to the random variables $X_1,\ldots,X_N$. To avoid technicalities, we assume that $\cE$ has been constructed by using the second option with $n=N$.

Let $\lambda_0^{t,N}, \ldots, \lambda_{N-1}^{t,N}$ and $\phi_0^{t,N} \ldots, \phi_{N-1}^{t,N}$ be the (random) eigenvalues and normalized eigenvectors of $\mathrm{L}_{t,N}$.

In order to establish the convergence in probability of the spectrum of $\mathrm{L}_{t,N}$ towards the one of $-\underline{\Delta}$, M.~Belkin and P.~Niyogi introduced two intermediary random operators:
\begin{itemize}
\item the \textit{point cloud Laplace operator} $$
\mathcal{L}_{t,N} : C(M) \to C(M) 
$$
defined by
$$
\mathcal{L}_{t,N}f(x) = \frac{1}{t}\frac{1}{(4\pi t)^{d/2}} \left( \frac{1}{N} \sum_{i=1}^N e^{-\frac{\|x-x_i\|_{\setR^D}^2}{4t}}(f(x)-f(x_i)) \right)
$$
for any $f \in C(M)$ and $x \in M$,
\item the \textit{Gaussian functional approximation} of $\underline{\Delta}$
$$
\mathfrak{L}_{t} : L^2(M) \to L^2(M)
$$
defined by
$$
\mathfrak{L}_{t}f(x)=\frac{1}{t}\frac{1}{(4\pi t)^{d/2}} \left( \int_M e^{-\frac{\|x-y\|_{\setR^D}^2}{4t}}(f(x)-f(y)) \di \underline{\vol}(y) \right)
$$
for any $f \in L^2(M)$ and $x \in M$.
\end{itemize}
The point cloud Laplace operator $\mathcal{L}_{t,N}$ acting on the Banach space $(C(M),\|\cdot\|_\infty)$ can be viewed as the difference between the multiplication operator $M_{t,N}$ and the finite rank operator $S_{t,N}$ defined by
$$
M_{t,N} f (x) = w_{t,N}(x) f(x) \qquad \text{and} \qquad  S_{t,N} f(x) = \frac{1}{t(4\pi t)^{d/2}} \frac{1}{N} \sum_{i=1}^N e^{-\frac{\|x-x_i\|_{\setR^D}^2}{4t}} f(x_i)
$$
for any $f \in C(X)$ and $x \in M$, where $w_{t,N}(x) := t^{-1}(4 \pi t)^{-d/2} N^{-1}  \sum_{i=1}^N e^{-\|x-x_i\|_{\setR^D}^2(4t)^{-1}}$. Therefore, the eigenvalue problem
\begin{equation}\label{eq:eigenvalueproblem}
\mathcal{L}_{t,N}f = \nu f
\end{equation}
admits a finite number of solutions.

 Since any function $f:M\to \setR$ defines a function $f_\cV:\cV\to\setR$ by setting $f_\cV(X_i) := f(X_i)$ for any $X_i \in \cV$, if $f$ is a solution of \eqref{eq:eigenvalueproblem}, then
$$
\mathrm{L}_{t,N}f_\cV = \nu f_{\cV}.
$$
From there, one gets that the spectrum of $\mathcal{L}_{t,N}$ coincides with the one of $L_{t,N}$, and that if $f$ is an eigenfunction of $\mathcal{L}_{t,N}$, then $f_{\cV}$ is an eigenfunction of $\mathrm{L}_{t,N}$. Thus it is sufficient to show the convergence of the spectrum of $\mathcal{L}_{t,N}$ towards the one of $-\underline{\Delta}$. 

Before stating the main theorem of this paragraph, let us point out that the eigenvalues $\{\underline{\lambda}_i\}_{i\ge0}$ of $-\underline{\Delta}$ satisfy
$$
\underline{\lambda}_i = \vol(M)\lambda_i
$$
for any $i$, where $\{\lambda_i\}_i$ are the eigenvalues of $-\Delta$, and that the set of orthonormal basis of $L^2(M)$ made of corresponding eigenfunctions coincides with $\cB(M,g)$.

\begin{theorem}
Let $\{\underline{\lambda}_i\}_{i}$ be the eigenvalues of $-\underline{\Delta}$ and $\{\phi_i\}_i \in \mathcal{B}(M,g)$. Then there exists an infinitesimal sequence $\{t_N\}_N \subset (0,+\infty)$ such that for any $i$, the following are true in probability when $N \to +\infty$:
$$
\lambda_{i}^{t_N,N} \to \underline{\lambda}_i \qquad \text{and} \qquad \|\psi_i^{t_N,N} - \phi_i\|_\infty \to 0,
$$
where the $\{\psi_i^{t,N}\}_i$ are normalized eigenfunctions of $\mathcal{L}_{t,N}$ for any $t>0$.
\end{theorem}

The proof of this theorem is made of two steps. We only explain here how to get the convergence result for the eigenvalues.\\

\textbf{Step 1.} The first step consists in showing that the spectrum of the point cloud Laplace operator $\mathcal{L}_{t,N}$  converges as $N$ goes to $+\infty$ to the one of the Gaussian functional approximation $\mathfrak{L}_{t}$. This is a direct consequence of a general theorem obtained by M.~Belkin, O.~Bousquet and U.~von Luxburg in \cite{BelkinBousquetvonLuxburg}. For completeness, let us cite this theorem.  The proof is a suitable application of Hoeffding's inequality. We use the notation $\sigma(T)$ to denote the spectrum of an operator $T$.

\begin{theorem}
Let $(X,\dist,\mu)$ be a compact probability metric measure space, $k:X\times X \to [0,+\infty)$ a continuous and symmetric map such that $k(x,y)>0$ for any $x\neq y$ in $X$, and $\{X_i\}_{i\ge 1}$ independant and identically distributed random variables on $X$ with same law $\mu$. For any $N \ge 1$, let $(\cV_N,\cE_N,k)$ be the random weighted graph constructed from $\{X_1,\ldots,X_N\}$ and $L_N$ the associated random matrix obtained by applying the Laplacian Eigenmaps process. Let $P_k:C(X)\to C(X)$ be the operator defined by
$$
P_k f(x)= \int_X k(x,y)(f(x) - f(y))\di \mu(y)
$$
for any $f \in C(X)$ and $x \in X$. Then for any $\lambda \in \sigma(P_k)\backslash \{1\}$ and any neighborhood $U \subset \setC$ of $\lambda$ that does not contain any other eigenvalue of $P_k$,
\begin{enumerate}
\item any sequence $\{\lambda_N\}_N$ such that $\lambda_N \in \sigma(L_N) \cap U$ for any $N$ satisfies
$$
\lambda_N \to \lambda \qquad \text{almost surely},
$$
\item if $\lambda$ is simple with associated normalized eigenfunction $\phi$ and $\{\lambda_N\}_N$ is such that $\lambda_N \in \sigma(L_N) \cap U$ with normalized associated eigenvector $\phi_N=(\phi_N(X_1),\ldots,\phi_N(X_N))$ for any $N$, then there exists a sequence $\{\eps_N\}_N \subset \{-1,1\}^{\setN}$ such that
$$
\sup_{1\le i \le N} |\eps_N \phi_N(X_i) - \phi(X_i)| \to 0 \qquad \text{almost surely}.
$$
\end{enumerate}
\end{theorem}

\begin{remark}
Theorem 15 in \cite{BelkinBousquetvonLuxburg} also contains a convergence result when $\lambda$ is not simple formulated in terms of spectral projections.
\end{remark}

\hfill

\textbf{Step 2.} The second step is to study the difference between $-\underline{\Delta}$ and its Gaussian functional approximation $\mathcal{L}_t$. This study is based on an elementary result.

\begin{lemma}
Let $H$ be a Hilbert space and $A,B$ two non-negative self-adjoint operators on $H$ with discrete spectrum listed in increasing order $\{\lambda_i(A)\}_i$ and $\{\lambda_i(B)\}_i$ respectively. Then for any $\eps>0$,
$$
\sup_{x \in H} \left| \frac{\langle (A-B)x, x \rangle}{\langle Ax, x\rangle} \right| \le \eps
$$
implies
$$
1-\eps \le \frac{\lambda_i(A)}{\lambda_i(B)}\le 1+\eps \qquad \text{for any $i$.}
$$
\end{lemma}

Thanks to this lemma, if we set $D_t:=(\Id-e^{-t\underline{\Delta}})/t$ for any $t>0$, then showing
\begin{equation}\label{eq:toprove2}
\sup_{f \in L^2(M)} \left| \frac{\langle (D_t-\mathfrak{L}_t)x, x \rangle}{\langle D_t x, x\rangle} \right|  = o(1) \qquad \text{when $t \downarrow 0$}
\end{equation}
implies
$$
\lambda_i(\mathfrak{L}_t) = \lambda_{i}(D_t) + o(1) \qquad \text{when $t \downarrow 0$}
$$
for any $i$, hence the desired result since $$\lambda_{i}(D_t) = (1-e^{-\underline{\lambda}_i t})/t \to \underline{\lambda}_i \qquad \text{when $t \downarrow 0$}.
$$

Let us explain how to prove \eqref{eq:toprove2}. Take $\alpha>0$ to be suitably chosen later, and $f \in L^2(M)$. Write $f = \sum_{i} a_i \phi_i$ and set
$$
f_1 := \sum_{\lambda_i \le \alpha} a_i \phi_i \qquad \text{and} \qquad f_2 := \sum_{\lambda_i > \alpha} a_i \phi_i.
$$
Straightforward manipulations based on the Cauchy-Schwarz inequality lead to
\begin{equation}\label{eq:estimate3}
\left| \frac{\langle (D_t-\mathfrak{L}_t)f, f \rangle}{\langle D_t f, f\rangle} \right| \le \frac{3\|(D_t - \mathfrak{L}_t)f_1\|_{L^2}}{|\langle D_t f, f\rangle|}  + \frac{\|(D_t - \mathfrak{L}_t)f_2\|_{L^2}\|f_2\|_{L^2}}{|\langle D_t f, f\rangle|}
\end{equation}
and the result follows from three estimates. The first one is 
$$
 \langle D_t \phi_i, \phi_i\rangle \ge \frac{1}{2} \min(\lambda_i,1/\sqrt{t})
$$
for any $i$, which is an easy  consequences of the concavity and monotonicity of $F:\lambda\mapsto (1-e^{-\lambda t})/t$. This estimate notably implies
\begin{equation}\label{eq:estimate1}
\langle D_t f, f\rangle \ge \frac{\lambda_1}{2} \qquad \text{and} \qquad 
\langle D_t f, f\rangle \ge \frac{1}{2}\min(\alpha,1/\sqrt{t}) \|f_2\|^2_{L^2}.
\end{equation}
The two others are
\begin{equation}\label{eq:estimate2}
\|(D_t - \mathfrak{L}_t)f_1\|_{L^2} \le C_1 \sqrt{t} \alpha^{\frac{d+2}{4}}
\qquad \text{and} \qquad
\|(D_t - \mathfrak{L}_t)f_2\|_{L^2} \le C_2 \|f_2\|_{L^2}
\end{equation}
where $C_1$ and $C_2$ depends on the submanifold $M$. The proof of these two estimates is too long to be described here, but we stress out the fact that it involves only classical tools from geometric analysis, like the change of variable with exponential coordinates, the Sobolev embedding $W^{\frac{d}{2}+1,2}(M) \hookrightarrow \Lip(M)$ and the Minakshisundaram-Pleijel expansion. Combining \eqref{eq:estimate3}, \eqref{eq:estimate1} and \eqref{eq:estimate2}, we get
$$
\left| \frac{\langle (D_t-\mathfrak{L}_t)f, f \rangle}{\langle D_t f, f\rangle} \right| \le C_3 (\sqrt{t}\alpha^{\frac{d+2}{4}} + \max(1/\alpha,\sqrt{t}))
$$
where $C_3>0$ depends on the submanifold $M$, so that choosing $\alpha=t^{-\frac{2}{d+6}}$, for instance, implies \eqref{eq:toprove2}.

\begin{remark}
For more convergence results including information on the convergence rate, we refer to \cite{TrillosGerlachHeinSlepcev}  and the references therein.
\end{remark}

\subsection{Singer-Wu Vector Diffusion Maps}

The second algorithm we present is due to A.~Singer and H.-T.~Wu \cite{SingerWu}. The rough idea is to compare two vectors in the data point cloud after having used an orthonormal transformation to make them as close as possible.\\

\textbf{The setting}

Let $x_1,\ldots,x_N \subset \setR^D$ be lying on a smooth $d$-dimensional submanifold $M$ of $\setR^D$. Set $X:=\{x_1,\ldots,x_N\}$.\\

\textbf{The algorithm}

Choose two decreasing functions $K_1,K_2:[0,+\infty)\to[0,+\infty)$ both supported in $[0,1]$\footnote{A.~Singer and H.-T.~Wu chose for $K_1$ the Epanechnikov kernel $K(u)=(1-u^2)\chi_{[0,1]}$ and for $K_2$ the Gaussian kernel $K(u)=\exp(-u^2)\chi_{[0,1]}$}. For any $\eps>0$, set
\[
K_\alpha^{\eps}(\cdot):=K_\alpha(\cdot/\sqrt{\eps})
\]
for any $\alpha \in \{1,2\}$ and  
\[
\mathcal{N}_i^{\eps}:=\{x \in X \, : \, \|x-x_i\|_{\setR^D} \le \sqrt{\eps}\},
\]
that is the set of $\sqrt{\eps}$-neighbors of $x_i$, for any $i \in \{1,\ldots,N\}$. Choose two numbers $0<\eps_1<\eps_2$ in such a way that $d \le \inf \{ |\mathcal{N}_i^{\eps}| \, : \, 1 \le i \le N \}$, so that the sets of neighbors $\mathcal{N}_i^{\eps}$ all contain at least $d$ elements.
\\

\textbf{Step 1:} Local PCA.

The goal of this step is to construct for any $x_i \in X$ a suitable family of orthonormal vectors $\{u_{i,1}, \ldots,u_{i,d}\} \subset \setR^D$ serving as an approximation of an orthonormal basis of $T_{x_i}M$. To simplify the presentation, let us write $m_i:=|\mathcal{N}_i^{\eps}|$ and $\mathcal{N}_i^{\eps}:=\{x_{j_1}, \ldots, x_{j_{m_i}}\}$. Define the matrix 
$$
A_i:=\left[ \lambda_{j_k} (x_{j_k}-x_i) \right]_{1 \le k \le m_i}
$$
where $\lambda_{j_k}:=\sqrt{K_1^{\eps_1}(\|x_{j_k}-x_i\|_{\setR^N})}$ for any $k$. Note that the columns of $A_i$ are the vectors $x_{j_k}-x_i$ rescaled by a factor $\lambda_{j_k}$ which gets big when $\|x_{j_k}-x_i\|_{\setR^N}$ is small: in this way, the closer a point is from $x_i$, the bigger the norm of the corresponding column of $A_i$ is. Compute the singular valued decomposition
$$
A_i=U_i D_i V_i^{*}
$$
and form the $D \times d$ matrix $O_i$ by selecting the $d$ first columns of the matrix $U_i$: if $U_i=[u_{i,1},\ldots,u_{i,D}]$, then
$$
O_i:=[u_{i,1},\ldots,u_{i,d}].
$$

\hfill

\textbf{Step 2:} Alignment.

The goal of this step is to provide, for any $x_i$ and $x_j$ close enough, an orthogonal matrix $O_{ij}$ serving as an approximation of the parallel transport operator between $T_{x_i}M$ and $T_{x_j}M$. This process is called \textit{alignment}. The parameter $\eps_2$ helps quantifying the proximity between $x_i$ and $x_j$. For any $x_i \in X$ and $x_j \in \mathcal{N}_{i}^{\eps_2}$, set
$$
O_{ij}:=\mathrm{argmin}\{\|O-O_i^{T}O_j\|_{HS} \, : \, O \in O(d)\}.
$$
The solution of this minimization problem is $O_{ij}=U_{ij}V_{ij}^{T}$ where $U_{ij}$ and $V_{ij}$ are provided by the singular valued decomposition of $O_{i}^TO_j$.\\

\textbf{Step 3:} Weighted graph.

Define from $X$ a graph $(\cV,\cE)$ by setting $\cV:=X$ and $\cE:=\{(x_i,x_j) \in X\times X \, : \, \|x_i - x_j\|_{\setR^D}<\sqrt{\eps_2} \}$. We use again the notation $j\sim i$ to mean that $(i,j) \in \cE$. Equip $(\cV,\cE)$ with the weight $w$ defined by $w_{ij}:=K_2^{\eps_2}(\|x_i - x_j\|_{\setR^D})$ for any $1 \le i,j \le N$. Here again points $x_j$ that are close to $x_i$ gets more importance than those that are far from $x_i$.\\

\textbf{Step 4:} Averaging operator.

Consider the $Nd \times Nd$ matrix $S$ made of $N\times N$ blocks $(S(i,j))_{1\le i,j \le N}$ of size $d \times d$, where
\[
S(i,j):=
\begin{cases}
(d_w(i))^{-1} w_{ij}O_{ij} & \text{when $i\sim j$,}\\
0 & \text{otherwise},
\end{cases}
\]
where $d_w(i):=\sum_{j\sim i}w_{ij}$. For any $v=(v(1),\ldots,v(N)) \in \setR^{Nd}$ where each $v(i)$ is a vector of $\setR^d$, the vector $Sv=(Sv(1),\ldots Sv(N))$ is such that
\[
(Sv)(i) = \frac{1}{\deg(i)} \sum_{j \sim i} w_{ij}O_{ij}v(j)
\]
for any $i$. Understanding each $v(i)$ as a vector in $T_{x_i}M$ and the vector $O_{ij}v(j)$ as an approximation of the parallel transport of $v(j) \in T_{x_j}M$ into $T_{x_i}M$, we see that the matrix $S$ acts as a local weighted averaging operator for vector fields -- local because it takes into account only the points $x_j$ in the $\sqrt{\eps_2}$-neighborhood of $x_i$.\\

\textbf{Step 5.} Vector diffusion mappings.

Set
$$
\tilde{S}:=D^{-1/2}SD^{1/2}
$$
where $D$ is the diagonal $dN\times dN$ matrix whose $i$-th diagonal $d\times d$ block is $d_w(i)I_d$. Since $\tilde{S}$ is symmetric, it admits eigenvalues $\lambda_1,\ldots,\lambda_{nd}$ and associated normalized eigenvectors $v_1,\ldots,v_{nd}$. We order the eigenvalues in decreasing order of modulus: $|\lambda_1|\ge |\lambda_2| \ge \ldots \ge |\lambda_{nd}|$. A direct computation shows that for any $k \in \setN$ and $1\le i,j \le nd$,
\begin{align*}
\|\tilde{S}^{2k}(i,j)\|_{HS}^2 & = \sum_{l,r=1}^{nd} (\lambda_l \lambda_r)^{2k}\langle v_l(i),v_r(i)\rangle \langle v_l(j),v_r(j)\rangle\\
& = \langle V_k(i),V_k(j) \rangle
\end{align*}
where we have set
$$
V_k(i):=((\lambda_l \lambda_r)^{k} \langle v_l(i),v_r(i) \rangle)_{1\le l,r\le nd}.
$$
Then for any $t>0$, the maps
$$
V_t : X \ni x_i \mapsto ((\lambda_l \lambda_r)^{k} \langle v_l(i),v_r(i) \rangle)_{1\le l,r\le nd}
$$
are called \textit{vector diffusion mappings} of $X$ and the maps
$$
[V_t]^{m} : X \ni x_i \mapsto ((\lambda_l \lambda_r)^{k} \langle v_l(i),v_r(i) \rangle)_{1\le l,r\le m}
$$
are their truncated analogues. These latter maps serve as embeddings of the data set into the Euclidean space $\setR^{m^2}$.

\begin{remark}
The connection between the vector diffusion mappings and the vector diffusion maps for closed Riemannian manifolds is established in \cite[Sect.~5]{SingerWu} in a similar fashion as the one between the Laplacian Eigenmaps and the Bérard-Besson-Gallot spectral embeddings.
\end{remark}

\bibliographystyle{alpha}
\bibliography{Survey_spectral_embeddings}

\end{document}